\documentclass[a4paper,12pt]{article}

\usepackage{amsmath}
\usepackage{amsthm}
\usepackage{amssymb}

\usepackage{epsfig}
\usepackage{psfrag}
\usepackage{subfigure}
\usepackage{graphicx}

\usepackage[mathscr,mathcal]{eucal}
\usepackage{enumerate}

\usepackage{t1enc}
\usepackage[latin2]{inputenc}

\def \E {\mathbb E}

\def \N {\mathbb N}

\def \P {\mathbb P}

\def \R {\mathbb R}

\def \Z {\mathbb Z}

\def\boP{\mathbf{P}}

\def\boV{\mathbf{V}}

\def\cB{\mathcal{B}}
\def\cC{\mathcal{C}}
\def\cD{\mathcal{D}}
\def\cE{\mathcal{E}}
\def\cF{\mathcal{F}}

\def\cH{\mathcal{H}}

\def\cK{\mathcal{K}}
\def\cL{\mathcal{L}}

\def\cO{\mathcal{O}}

\def\cS{\mathcal{S}}

\def\cV{\mathcal{V}}

\def\vareps{\varepsilon}

\newcommand{\prob}[1]{\ensuremath{\mathbf{P}\big(\,#1\,\big)}}
\newcommand{\expect}[1]{\ensuremath{\mathbf{E}\big(\,#1\,\big)}}

\newcommand{\condprob}[2]{\ensuremath{\mathbf{P}\big(\,#1\,\big|\,#2\,\big)}}
\newcommand{\condexpect}[2]{\ensuremath{\mathbf{E}\big(\,#1\,\big|\,#2\,\big)}}

\newcommand{\ind}[1]{\ensuremath{{1\!\!1}_{\{#1\}}}}

\def \toprob {\,\,\buildrel\boP\over\longrightarrow\,\,}

\newcommand{\vct}[1]{ \text{\boldmath{$\mathrm{#1}$}} }

\def \Ordo {\cO}
\def\ordo{o}

\def\lan{\langle}
\def\ran{\rangle}

\def\ps{\partial_s}
\def\px{\partial_x}

\def\pss{\partial^2_{ss}}

\def\one{1\!\!1}

\newcommand{\abs}[1]{\left|{#1}\right|}

\def \wt {\widetilde}
\def\ol{\overline}

\def\wh{\widehat}

\def\ot{\ol{t}}

\newtheorem {theorem}{Theorem}

\newtheorem {lemma}{Lemma}

\newtheorem {proposition}{Proposition}

\newtheorem {definition}{Definition}

\newtheorem* {theorem*}{Theorem}
\newtheorem* {thm*}{Theorem}
\newtheorem* {lemma*}{Lemma}
\newtheorem* {lem*}{Lemma}
\newtheorem* {corollary*}{Corollary}
\newtheorem* {cor*}{Corollary}
\newtheorem* {proposition*}{Proposition}
\newtheorem* {prop*}{Proposition}
\newtheorem* {definition*}{Definition}
\newtheorem* {def*}{Definition}
\newtheorem* {conjecture*}{Conjecture}
\newtheorem* {remark*}{Remark}
\newtheorem* {rem*}{Remark}

\def\be{\begin{equation}}
\def\ee{\end{equation}}

\def\bea{\begin{eqnarray}}
\def\eea{\end{eqnarray}}

\def \barn {\bar\N}
\def \weak {\Rightarrow}
\def \rbeps {\varepsilon}
\def \rbt {\tilde{t}}
\def \spect {t^*}
\def \rbch {\xi}
\def \rbu {u}
\def \rbtfin {\bar{t}}
\def \rbV {\mathbf{V}}
\def \rbE {\mathbf{E}}

\title{Erd\H os-R\'enyi random graphs + forest fires = self-organized criticality}

\author{
{\sc Bal\'azs R\'ath} \qquad and  \qquad {\sc  B\'alint T\'oth}
\\[5pt]
Institute of Mathematics \\ Budapest University of Technology (BME)
\\
{\sc
Egry J\'ozsef u. 1}
\\
{\sc H-1111 Budapest, Hungary}
\\
e-mail: {\tt \{rathb,balint\}{@}math.bme.hu}
}

\begin{document}

\date{\today}

\maketitle

\bigskip

\begin{abstract}
We modify the usual Erd\H os-R\'enyi random graph evolution by letting
connected clusters 'burn down' (i.e. fall apart to disconnected single
sites) due to a Poisson flow of lightnings. In a range of the
intensity of rate of lightnings the system sticks to a permanent
critical state.
\end{abstract}

$ $


$ $


$ $

\noindent
{\bf Acknowledgement.}
This research was  partially supported by the OTKA (Hungarian
National Research Fund) grants
K 60708 and
TS 49835. The authors thank the anonymous referee for thoroughly reading the manuscript. Her/his comments and suggestions helped us improving the presentation.

$ $

\noindent
Submitted to EJP on September 1, 2008, final version accepted May 14, 2009.
\newpage

\tableofcontents

\section{Introduction}
\label{section:introduction}

\subsection{Context}
\label{subsection:context}

In conventional models of equilibrium statistical physics, such as
Bernoulli percolation, random cluster models, the Ising model or
the Heisenberg model there is always a parameter which controls
the character of the equilibrium Gibbs measure: in percolation and
random cluster-type models this is the density of open
sites/edges, in the Ising or Heisenberg models the inverse
temperature. Typically the following happens: tuning the control
parameter at a particular value (the critical density or the
critical inverse temperature) the system exhibits critical
behavior in the thermodynamical limit, manifesting e.g. in power
law rather than exponential decay of the upper tail of the
distribution of the size of connected clusters. Off this
particular critical value of the control parameter these
distributions decay exponentially. We emphasize here that the
critical behavior is observed only at this particular critical
value of the control parameter.

As opposed to this, in some dynamically defined models of
interacting microscopic units one expects the following robust
manifestation of criticality: In some systems dynamics defined
naturally in terms of local interactions some effects can
propagate  instantaneously through macroscopic distances in the
system. This behavior may have dramatic effects on the global
behavior, driving the system to a permanent critical state. The
point is that without tuning finely some parameter of the
interaction the dynamics drives the system to criticality. This
kind of behavior is called \emph{self-organized criticality
  (SOC)} in the physics literature.
The two best known examples are the sandpile models where so
called avalanches spread over macroscopic distances
instantaneously, and the forest fire models where beside the
Poissonian flow of switching sites/edges from ``empty'' to
``occupied'' state (i.e. trees being grown), at some instants
connected clusters of occupied sites/edges (forests of trees) are
turned from ``occupied'' to ``empty'' state instantaneously  (i.e.
forests hit by lightnings are burnt down on a much faster time
scale than the growth of trees). These models and these phenomena
prove to be difficult to analyze mathematically rigorously due to
the following two facts: (1) There are always two competing
components of the dynamics (in the forest fire models: growing
trees and burning down forests) causing lack of any kind of
monotonicity of the models. (2) Long range effects due to
instantaneous propagation of short range interactions are very
difficult to be controlled.

Regarding forest fire models there are very few mathematically
rigorous results describing SOC. The best known and most studied
model of forest fires is the so-called Drossel-Schwabl model. For
the original formulation  see \cite{drossel_schwabl}, or the more
recent survey \cite{schenk_drossel_schwabl}. We formulate here a
related variant.

Let $\Lambda_n:=\Z^d\cap[-n,n]^d$.  The state space of the model
of size $n$ is $\Omega_n:=\{0,1\}^{\Lambda_n}$: sites of
$\Lambda_n$ can be occupied by a tree (1) or empty (0). The
dynamics consists of two competing mechanisms:
\\
(A) Empty (0) sites turn occupied (1) with rate one, independently
of whatever else happens in the system.
\\
(B) Sites get hit by ``lightnings'' with rate $\lambda(n)$,
independently of whatever else happens in the system. When site is
hit by lightning its whole connected cluster of occupied sites
turns instantaneously from ``occupied'' (1) to ``empty'' (0)
state. (That is: when a tree is hit by lightning the whole forest
to which it belongs burns down instantaneously.)
\\
The dynamics goes on indefinitely.

As long as $n$ is kept fixed the  mechanism $A+B$ defines  a
decent  finite state Markov process -- though a rather complicated
one. The main question is: what happens in the thermodynamic
limit, when $n\to\infty$, $\Lambda_n\nearrow\Z^d$? Can one specify
a dynamics on the state space $\Omega_{\infty}:=\{0,1\}^{\Z^d}$
which could be identified with the infinite volume limit of the
systems defined above?

In order to make some guesses, one has first to specify the
lightning rate $\lambda(n)$. Intuitively one expects four regimes
of the rate $\lambda(n)$ with essentially different asymptotic
behavior of the system in the limit of infinite volume:

\begin{enumerate} [I.]

\item
If $\lambda(n)\ll|\Lambda_n|^{-1}$ then the effect of lightning is
simply not felt in the thermodynamic limit: in macroscopic time
intervals of any fixed length no lightning will hit the entire
system. Thus, in this regime the system will simply be the
dynamical formulation of Bernoulli percolation.

\item
If $\lambda(n)=|\Lambda_n|^{-1} \lambda$ with some  fixed
$\lambda\in(0,\infty)$ then one expects in the thermodynamic limit
the following dynamics (described in plain, non-technical terms).
The system evolves as dynamical site percolation, with independent
Poisson evolutions on sites,  and with rate $\lambda \theta(t)$,
where $\theta(t)$ is the density of the (unique) infinite cluster,
the sites of this (unique) infinite cluster  are turned from occupied to empty.
 After this forest fire the system keeps on evolving like dynamical percolation
 until a new infinite component is born,  and the dynamics goes on indefinitely.

\item
If $|\Lambda_n|^{-1} \ll \lambda(n)\ll1$ then in the infinite
volume limit - if it makes any sense - something really
interesting must happen: The lightning rate is too small to hit
finite clusters within any finite horizon. But it is too large to
let the infinite percolating cluster to be born. One can expect
(somewhat naively) that in this  regime in the thermodynamic limit
a dynamics will be defined on $\Omega_{\infty}$ in which \emph{in
plain words} the following happens:
\\
- empty (0) sites turn occupied (1) with rate one, independently
of whatever else happens in  the system;
\\
- when the \emph{incipient infinite percolating cluster} is about
to be born, it is switched from ``occupied'' (1) to ``empty'' (0)
state;
\\
- the dynamics goes on indefinitely.
\\
In this way this presumed infinitely extended dynamics would stick
to a permanent critical state when the infinite incipient critical
cluster is always about to be born, but not let to grow beyond
criticality.

\item
If $\lambda(n)=\lambda\in(0,\infty)$ then lightning will hit
regularly even small clusters and thus, one may expect that - if
the infinitely extended dynamics is well defined - the system will
stay subcritical indefinitely.

\end{enumerate}

There is no problem with the mathematically rigorous definition of
the infinitely extended dynamics in regimes I. and II. But these
plain descriptions don't necessarily make mathematical sense
 and  it is not at all clear that such
infinitely extended critical forest fire models can at all be defined in a mathematically satisfactory
way.

In our understanding, the most interesting open questions are the
existence and characterization of the infinitely extended dynamics
in regime III. and/or the $\lambda\to\infty$ limit in regime II.
and/or the $\lambda\to0$ limit in regime IV., after the
thermodynamic limit.

There are however some deep results regarding these (or some other
related) models of forest fires, though clarification of the above
questions seems to be far out of reach at present.

Here follows a (necessarily incomplete) list of some important
results related to these questions:

\begin{enumerate}[--]

\item
M. D{\"u}rre proves existence of infinitely extended forest fire
dynamics in a related model  in the subcritical regime IV. ,
\cite{durre_1}. In a companion paper he also proves that under
some regularity conditions assumed the dynamics is uniquely
defined, \cite{durre_2}.

\item
J. van den Berg and R. Brouwer, respectively R. Brouwer consider
the so called \emph{self-destructive percolation} model, which is
very closely related to what we called regime II. above. They
prove various deep technical results and formulate some intriguing
conjectures related to the $\lambda\to\infty$ limit in regime II. (of the
already infinitely extended dynamics), see
\cite{vandenberg_brouwer_1}, \cite{vandenberg_brouwer_2},
\cite{brouwer_2}

\item
J. van den Berg and A. J\'arai analyze the $\lambda\to 0$
asymptotics of the (infinitely extended) model in regime IV. in
dimension 1, \cite{vandenberg_jarai}.

\item
J. van den Berg and B. T\'oth consider an \emph{inhomogeneous} one
dimensional model which indeed exhibits SOC, see
\cite{vandenberg_toth}. (In one dimensional space-homogeneous
models of course there is no critical behavior)
\end{enumerate}

\subsection{The model}
\label{subsection:model}

We investigate  a modification of the dynamical formulation of the
Erd\H os-R\'enyi random graph model, adding ``forest fires''
caused by  ``lightning'' to the conventional Erd\H os-R\'enyi
coagulation mechanism. Actually our model will be a particular
coagulation-fragmentation dynamics exhibiting robust
self-organized criticality.

Let $\cS_n:=\{1,2,\dots,n\}$ and $\cB_n:=\{(i,j)=(j,i):
i,j\in\cS_n, \, i \neq j \}$ be the set of vertices, respectively, unoriented
edges of the complete graph $\cK_n$. We define a dynamical random
graph model as follows. The state space of our Markov process is
$\{0,1\}^{\cB_n}$.

 Edges $(i,j)$ of $\cK_n$ will be
called occupied or empty according whether $\omega(i,j)=1$ or
$\omega(i,j)=0$.  As usual, we call clusters the maximal subsets
connected by occupied edges.

Assume that initially, at time $t=0$,  all edges are empty. The
dynamics consists of the following

\begin{enumerate} [(A)]

\item
Empty edges turn occupied with rate $1/n$, independently of
whatever else happens in the system.

\item
Sites of $\cK_n$ get hit by lightnings with rate $\lambda(n)$,
independently of whatever else happens in the system. When a site
is hit by lightning, all edges which belong to its connected
occupied cluster turn instantaneously empty.

\end{enumerate}

In this way a random graph dynamics is  defined. The coagulation
mechanism (A) alone defines the well understood Erd\H os-R\'enyi
random graph model.
For basic facts and refined details of the Erd\H os-R\'enyi random graph
problem see \cite{erdos_renyi}, \cite{bollobas},
\cite{janson_luczak_rucinski}.
As we shall see soon, adding the fragmentation
mechanism (B)  may cause essential changes in the behavior of the
system.

We are interested of course in the asymptotic behavior of the
system when  $n\to\infty$. In order to formulate our problem first
have to introduce  the proper spaces on which our processes are
defined.

We denote
\begin{align}\label{def_V_configurations}
\cV :=& \big\{ \vct{v} = \big(v_k\big)_{k\in\N}\,:\, \, v_k\ge0, \
\ \  \sum_{k\in\N} v_k \le 1\big\}, \quad \theta(\vct{v}) :=1-
\sum_{k\in\N} v_k,
\\
\label{def_V_1_configurations}
\cV_1 :=& \big\{ \vct{v} \in \cV \,:\,  \theta(\vct{v}) = 0\big\}.
\end{align}
We endow $\cV$  with the (weak) topology of component-wise
convergence. We may interpret $\theta$ as
the density of the giant component.

A map $[0,\infty)\ni t\mapsto \vct{v}(t)\in\cV$ which is
component-wise of bounded variation on compact intervals of time
and continuous from the left in $[0,\infty)$, will be called
\emph{a forest fire evolution (FFE)}. If
$\vct{v}(t)\in\cV_1$  for all $t\in[0,\infty)$ we call the FFE
\emph{conservative}. Denote the space of FFE-s and  conservative
FFE-s  by $\cE$, respectively, $\cE_1$. The space $\cE$ is endowed
with the topology of component-wise weak convergence of the signed
measures corresponding to the functions $v_k(\cdot)$ on compact
intervals of time. This topology is metrizable and the space $\cE$
endowed with this topology is complete and separable.

Now, we define the \emph{cluster size distribution} in our random
graph process as follows
\begin{align}
\label{def_of_v_k_in_markov}
v_{n,k}(t)
& :=
n^{-1} \#\{j\in\cS_n: j
\text{ belongs to a cluster of size } k
\text{ at time } t \}=:n^{-1}V_{n,k}(t),
\\[8pt]
\label{def_of_vct_v_in_markov}
\vct{v}_n(t)
& :=
\big(v_{n,k}(t)\big)_{k\in\N}.
\end{align}
This means that $\vct{v}_{n}(t)$ is the cluster size distribution
of a uniformly selected site from $\cS_n$, at time $t$. Clearly,
the random trajectory $t\mapsto\vct{v}_n(t)$ is a (conservative)
FFE. We consider the left-continuous version of $t\mapsto\vct{v}_n(t)$
instead of the traditional c.\`a.d.l.\`a.g., for technical reasons discussed in Subsection
\ref{subsection:fff}.

 We investigate the asymptotics of this process, as
$n\to\infty$.

It is well known (see e.g. \cite{buffet_pule_1},
\cite{buffet_pule_2}, \cite{aldous_1}) that in the Erd\H os-R\'enyi
case  --  that is: if $\lambda(n)=0$
\begin{equation}
\label{limprob}
\vct{v}_{n}(\cdot)\toprob \vct{v}(\cdot)
=\big(v_{k}(\cdot)\big)_{k\in\N} \quad \text{as $\quad n \to \infty$},
\end{equation}
where the deterministic functions $t\mapsto v_k(t)$ are solutions
of the infinite system of ODE-s
\begin{equation}
\label{smolu1}
\dot v_k(t) = \frac{k}{2} \sum_{l=1}^{k-1} v_{l}(t)
v_{k-l}(t) - k v_k(t), \qquad k\ge1,
\end{equation}
with initial conditions
\begin{equation}
\label{inimono}
v_k(0)=\delta_{k,1}.
\end{equation}
The infinite system of ODE-s \eqref{smolu1} are the
\emph{Smoluchowski coagulation equations}, the initial conditions \eqref{inimono} are
usually called \emph{monodisperse}. The system \eqref{smolu1} is
actually not very scary: it can be solved one-by-one for
$k=1,2,\dots$ in turn. For the initial conditions
\eqref{inimono}  the solution is known explicitly:
\[
v_k(t)=\frac{k^{k-1}}{k!}e^{-kt}t^{k-1}.
\]
$\left( v_k(t)\right)_{k=1}^{\infty} \in \cV$ is a (possibly defected) probability distribution
called the Borel distribution: in a Galton-Watson branching process with offspring distribution $POI(t)$ the resulting random tree has $k$ vertices with probability $v_k(t)$. Thus the branching process is subcritical, critical and supercritical for $t<1$, $t=1$ and $t>1$, respectively.

For general initial conditions $v_k(0)$ satisfying
\begin{equation*}
\label{inimoms}
\sum_{k=1}^\infty v_k(0)=1, \qquad
\sum_{k=1}^\infty k^2 v_k(0)<\infty,
\end{equation*}
the qualitative behavior of the solution of \eqref{smolu1} is
similar: Define the \emph{gelation time}
\begin{equation}\label{def_gel_time}
T_{\text{gel}}:=\big(\sum_{k=1}^\infty kv_k(0)\big)^{-1}
\end{equation}

\begin{enumerate} [--]

\item
For $0\le t< T_{\text{gel}}$ the system is subcritical:
$\theta(\vct{v}(t))=0$ and, $k\mapsto v_k(t)$ decay exponentially
with $k$.

\item
For $T_{\text{gel}}<t<\infty$ the system is supercritical:
$\theta(\vct{v}(t))>0$ and $k\mapsto v_k(t)$ decay exponentially
with $k$.  Further on: $t\mapsto\theta(\vct{v}(t))$ is smooth and
strictly increasing with  $\lim_{t\to\infty}
\theta(\vct{v}(t))=1$.

\item
Finally, at  $t= T_{\text{gel}}$ the system is critical:
$\theta(\vct{v}(T_{\text{gel}}))=0$ and
\begin{equation}
\label{erdos_renyi_critical_exponent}
 \sum_{l=k}^{\infty} v_l(T_{\text{gel}})\asymp
k^{-1/2} \quad \text{as $\quad k\to\infty$}.
\end{equation}

\end{enumerate}

Our aim is to understand in similar terms the asymptotic behavior
of the system when, beside the Erd\H os-R\'enyi coagulation
mechanism, the fragmentation due to forest fires also take place.

Similarly to the Drossel-Schwabl case presented in subsection
\ref{subsection:context} we have four regimes of the lightning
rate $\lambda(n)$, in which the asymptotic behavior is different:
\begin{align*}
& \text{I.:} && \lambda(n)\ll n^{-1},
\\[8pt]
& \text{II.:} && \lambda(n)=n^{-1} \lambda, \qquad
\lambda\in(0,\infty),
\\[8pt]
& \text{III.:} && n^{-1} \ll \lambda(n)\ll1,
\\[8pt]
& \text{IV.:} && \lambda(n)=\lambda\in(0,\infty).
\end{align*}

The $n\to\infty$ asymptotics of the processes $t\mapsto
\vct{v}_n(t)$ in the four regimes is summarized as follows:

\begin{enumerate} [I.]

\item
The effect of lightnings is simply not felt in the $n\to\infty$
limit. In this regime the system will  be the dynamical
formulation of the Erd\H os-R\'enyi random graph model, the
asymptotic  description presented in the previous paragraph is
valid.

\item
In the $n\to\infty$ limit the sequence of processes
$t\mapsto\vct{v}_{n}(t)$ converges weakly (in distribution) in the
topology of the space $\cE$ to a process $t\mapsto\vct{v}(t)$
described as follows: The process  $t\mapsto\vct{v}(t)$ evolves
deterministically, driven by the Smoluchovski equations
\eqref{smolu1} (exactly as in the limit of the dynamical Erd\H
os-R\'enyi model) with the following Markovian random jumps added
to the dynamics:
\begin{align}\label{giant_lightning}
&
\condprob{\vct{v}(t+dt)=J\vct{v}}{\vct{v}(t)=\vct{v}} = \lambda
\theta(\vct{v})dt+\ordo(dt)
\\[8pt] \label{giant_jump}
\text{where }\quad
&
J:\cV\to\cV, \quad
(J\vct{v})_k=v_k+\delta_{k,1}\theta(\vct{v}).
\end{align}
In plain words: with rate $\lambda \theta(\vct{v}(t))$ the amount
of mass $\theta(\vct{v}(t))$ contained in the gel (i.e. the unique
giant component) is instantaneously pushed into the singletons.

\item
This is the most interesting regime and \emph{technically the
content of the present paper}. In the $n\to\infty$ limit
\eqref{limprob} holds, where now the deterministic functions
$t\mapsto v_k(t)$ are solutions of the infinite system of
\emph{constrained ODE-s}
\begin{align}
\label{smolu2}
&
\dot v_k(t) = \frac{k}{2} \sum_{l=1}^{k-1}
v_{l}(t) v_{k-l}(t) - k v_k(t), \qquad k\ge2,
\\[8pt]
\label{nogel}
&
\sum_{k\in\N}v_k(t)=1,
\end{align}
with the initial conditions \eqref{inimono}. Mind the difference
between the system \eqref{smolu1} at one hand and the constrained
system \eqref{smolu2}+\eqref{nogel} at the other: the first
equation from \eqref{smolu1} is replaced by the global constraint
\eqref{nogel}. A first consequence is that it is no more true that
the ODE-s in \eqref{smolu2} can be solved for $k=1,2,\dots$,
one-by-one,  in turn. The system of ODE-s is \emph{genuinely
infinite}. Up to $T_{\text{gel}}$ the solutions of \eqref{smolu1},
respectively, of \eqref{smolu2}+\eqref{nogel} coincide, of course.
But dramatic differences arise beyond this critical time. We prove
that the system \eqref{smolu2}+\eqref{nogel} admits a \emph{unique
solution} and for $t\ge T_{\text{gel}}$
\begin{equation}
\label{largekasymp}
\sum_{l=k}^\infty v_l(t) \sim
\sqrt{\frac{2\varphi(t)}{\pi}} k^{-1/2}, \quad \text{ as } \quad
k\to\infty,
\end{equation}
where $[T_{\text{gel}},\infty)\ni t\mapsto \varphi(t)$ is strictly
positive, bounded and Lipschitz continuous. This shows that
in this regime the random graph dynamics exhibits indeed
\emph{self-organized critical behavior}: beyond the critical time
$T_{\text{gel}}$ it stays critical for ever. The unique stationary
solution of the system \eqref{smolu2}+\eqref{nogel} is easily found
\begin{equation}
\label{critstaci}
v_k(\infty)= 2 \binom{2n-2}{n-1} \frac{1}{n}
4^{-n} \approx \frac{1}{\sqrt{4\pi}}k^{-3/2}.
\end{equation}

\item
In the $n\to\infty$ limit \eqref{limprob} holds again, where now
the deterministic functions $t\mapsto v_k(t)$ are solutions of
the infinite system of ODE-s
\begin{equation}
\label{smolu3}
\dot v_k(t)=
\frac{k}{2}\sum_{l=1}^{k-1}v_l(t)v_{k-l}(t) - kv_k(t)  -\lambda  kv_k(t)
+\lambda \delta_{k,1} \sum_{l=1}^\infty l v_l(t), \qquad k\ge1,
\end{equation}
with the initial conditions in $\cV_1$. The system \eqref{smolu3}
is again a genuine infinite system (it can't be solved one-by-one
for $k=1,2,\dots$ in turn). The Cauchy problem \eqref{smolu3} with
initial condition in $\cV_1$ has a unique solution, which stays
\emph{subcritical}, i.e. for any $t\in(0,\infty)$
 $k\mapsto v_k(t)$ decays exponentially. The unique stationary solution is
closely related to that of \eqref{critstaci}:
\begin{equation*}
\label{subcritstaci}
v_{\lambda,k}(\infty)=
(\lambda+1)\left(1-\frac{\lambda^2}{(1+\lambda)^2} \right)^k
v_k(\infty)
\end{equation*}

\end{enumerate}

\subsection{The main results}
\label{subsection:results}

We present the results formulated and proved only for the regime
III: $n^{-1}\ll\lambda(n)\ll 1$, which shows \emph{self-organized
critical} asymptotic behaviour. The methods developed along the
proofs are sufficient to prove the asymptotic behaviour in the
other regimes, described in items I, II and IV but we omit these
(in our opinion less interesting) details.

\begin{theorem}
\label{theorem:uniqueness}
If the initial condition $\vct{v}(0)\in\cV_1$ is such that
$\sum_{k=1}^{\infty} k^3 v_k(0)<+\infty$, and $T_{\text{gel}}$ is defined by \eqref{def_gel_time}
 then the critical
forest fire equations \eqref{smolu2}+\eqref{nogel} have a unique
solution with the following properties:
\begin{enumerate}

\item
For $t \leq T_{\text{gel}}$ the solution coincides with that of
\eqref{smolu1}.

\item
For $t \geq T_{\text{gel}}$ there exists a positive, locally
Lipschitz-continuous
function $\varphi$ such that
\begin{equation}
\label{v1_evolution}
\dot{v}_1(t)=-v_1(t)+\varphi(t)
\end{equation}
and \eqref{largekasymp} holds.

\end{enumerate}
\end{theorem}

\begin{theorem}
\label{theorem:weak_conv_to_unique}
Let $\P_n$ denote the law of the random FFE of the forest fire
Markov chain $\vct{v}_n(t)$ with initial condition $\vct{v}_n(0)$
and lightning rate parameter $n^{-1}\ll\lambda(n)\ll1$. If
$\vct{v}_n(0) \to \vct{v}(0)\in\cV_1$ component-wise where
$\sum_{k=1}^{\infty} k^3 v_k(t)<+\infty$ then the sequence of
probability measures $\P_n$ converges weakly to the Dirac measure
concentrated on the unique solution of the critical forest fire
equations \eqref{smolu2}+\eqref{nogel} with initial condition
$\vct{v}(0)$. In particular
\[
\forall \varepsilon>0, \, t \geq 0 \quad \lim_{n \to \infty}
\prob{ \abs{v_{n,k}(t)-v_k(t)} \geq \varepsilon }=0
\]
\end{theorem}

\section{Coagulation and fragmentation}
\label{section:coagfrag}

\subsection{Forest fire flows}
\label{subsection:fff}

In this section we investigate the underlying structure of forest fire evolutions arising from the coagulation-fragmentation dynamics of our model on $n$ vertices.

We define auxiliary objects called forest fire flows:
 let $q_{n,k,l}(t)$ denote $n^{-1}$ times the number of $(k,l)$-coagulation events (a component of size $k$ merges with a component of size $l$) up to time $t$. Let $r_{n,k}(t)$ denote $n^{-1}\cdot k$ times the number of $k$-burning events (a component of size $k$ burns) up to time $t$. For the precise definitions see \eqref{def_of_Qnkl}, \eqref{def_of_Rnk}, \eqref{markov_flow_coag} and \eqref{markov_flow_fire}.

 In Subsection \ref{subsection:fff} and Subsection \ref{subsection:Markov_process} we precisely formulate and prove lemmas based on the following heuristic ideas:

\begin{itemize}

\item
The state $\vct{v}_n(t)$ of the forest fire process on $n$ vertices (see \eqref{def_of_vct_v_in_markov}) can be
recovered if we know the initial state $\vct{v}_n(0)$, and the flow: $q_{n,k,l}(t)$ for all $k,l \in \N$ and $r_{n,k}(t)$ for all $k$. The precise formula is \eqref{floweqs}.

\item \eqref{floweqs} is similar to the equations \eqref{smolu3}. This will help us proving Theorem \ref{theorem:weak_conv_to_unique}: if $1 \ll n$ and
$n^{-1} \ll \lambda(n) \ll 1$
 then the random forest fire evolution $\vct{v}_n(t)$ "almost" satisfies the equations \eqref{smolu2}+\eqref{nogel}  that uniquely determine the deterministic limiting object $\vct{v}(t)$. We essentially prove that \eqref{smolu2} is satisfied in the $n \to \infty$ limit in Proposition \ref{propo:weaklimit} of Subsection \ref{subsection:Markov_process}. We prove that \eqref{nogel} is satisfied in the limit in Subsection \ref{subsection:nogiant}.

\end{itemize}

$ $

We define the moments of $\vct{v} \in \cV$ as
\[ m_0=\sum_{k=1}^{\infty} v_k, \quad
 m_1=\sum_{k=1}^{\infty} k\cdot v_k,  \quad
  m_2=\sum_{k=1}^{\infty} k^2 \cdot v_k, \quad
 m_3=\sum_{k=1}^{\infty} k^3 \cdot v_k \]
By \eqref{def_V_configurations} and \eqref{def_V_1_configurations} $m_0=1$ if and only if $\vct{v} \in \cV_1$.

Fix $T\in(0,\infty)$. A map $[0,T]\ni t\mapsto \vct{v}(t)\in\cV$ is a
\emph{a forest fire evolution (FFE)} on
$\lbrack 0, T \rbrack$ if  $v_k(\cdot),$ $k \in \N$ is of
bounded variation and continuous from the left in $(0,T]$.
Denote the
space of FFE-s on $\lbrack 0, T \rbrack$  by $\cE[0,T]$ and the
space of FFE-s with initial condition $\vct{v}(0)=\vct{v}\in\cV$
on $\lbrack 0, T \rbrack$ by $\cE_{\vct{v}}[0,T]$.
Note that a priori $\theta(\cdot)=1-\sum_{k \in \N} v_k(\cdot)$
need not be of bounded variation.

If $\vct{v}_n(\cdot)\in \cE[0,T]$ is a sequence of FFE-s then we
say that $\vct{v}_n(\cdot) \to \vct{v}(\cdot)$ if
$v_{n,k}(\cdot) \weak v_k(\cdot)$ for all $k \in \N$ where ``$\weak$'' denotes  weak
convergence of the finite signed measures on $\lbrack 0 , T \rbrack$
corresponding to the functions $v_{n,k}(\cdot)$ and $v_k(\cdot)$.
Note that we did not require the convergence of $\theta_n(\cdot)$ to
$\theta(\cdot)$.

   This topology is metrizable and the spaces $\cE[0,T]$  and
$\cE_{\vct{v}}[0,T]$ endowed with this topology are  separable and
complete (by Fatou's lemma, $\lim_{n \to \infty} \vct{v}_n(t)$ stays in $\cV$).

Denote $\N:=\lbrace 1,2,\dots\rbrace$ and $\barn:=\N \cup
\lbrace \infty \rbrace$.

A \emph{forest fire flow (FFF)} is a collection of maps
$[0,T]\ni t\mapsto\big(\vct{q}(t), \vct{r}(t) \big)$
where for $0\le s\le t \le T$
\begin{align*}
&
0=q_{k,l}(0)\le q_{k,l}(s)\le q_{k,l}(t),
&&
\vct{q}(t)=\big(q_{k,l}(t)\big)_{k,l\in\barn},
&&
q_{k,l}(t)=q_{l,k}(t),
\\[8pt]
&
0=r_{k}(0)\le r_{k}(s)\le r_{k}(t),
&&
\vct{r}(t)=(r_{k}(t))_{k\in\barn},
&&
r_1(t)\equiv0
\end{align*}
We define
\begin{equation}\label{def_of_q_r}
q_k(t):=\sum_{l\in\barn}q_{k,l}(t),
\qquad
q(t):=\sum_{k\in\barn}q_{k}(t),
\qquad
r(t):=\sum_{k\in\barn}r_{k}(t)
\end{equation}
and assume the \emph{finiteness conditions}
$q(T)<+\infty$,
$r(T)<+\infty$. All functions involved are continuous from the left in $(0,T]$. This is why we have chosen
 to consider the left-continuous versions of these functions rather than the traditional
 c.\`a.d.l.\`a.g.: the supremum of increasing left-continuous functions is itself left-continuous, thus the left-continuity of $q_k$, $q$ and $r$ automatically follows from the left-continuity of $q_{k,l}$ and $r_k$.

 We say that the FFF $[0,T]\ni t\mapsto\big(\vct{q}(t), \vct{r}(t) \big)$
is \emph{consistent with the initial condition}
$\vct{v}(0)=\vct{v}\in\cV$
if $t\mapsto\vct{v}(t)$ defined by
\begin{equation}
\label{floweqs}
v_k(t)
=
v_k(0)
+
\frac{k}{2} \sum_{l=1}^{k-1}q_{l,k-l}(t)
-
k q_k(t)
-
 r_k(t)
+
\ind{k=1}r(t),
\hskip1cm
k\in \N.
\end{equation}
is in $\cE_{\vct{v}}[0,T]$. That is: for all $t\in[0,T]$ and $k
\in \N$ $v_k(t)\geq 0$ and $\sum_{k \in \N} v_k(t) \leq 1$ holds.
In this case we say that the FFF $\big(\vct{q}(\cdot),
\vct{r}(\cdot) \big)$ \emph{generates} the FFE $\vct{v}(\cdot)$.

 We denote by $\cF_{\vct{v}}[0,T]$ the space of FFF-s consistent
with the initial condition $\vct{v}(0)=\vct{v}\in\cV$. For any
$\vct{v}\in\cV$, $\cF_{\vct{v}}[0,T]\not=\emptyset$, since the
zero flow is consistent with any initial condition.

 At this point we mention that later we are going to obtain a FFF $\big(\vct{q}_n(\cdot), \vct{r}_n(\cdot) \big)$
 from a realization of our model on $n$ vertices by \eqref{def_of_Qnkl}, \eqref{def_of_Rnk}, \eqref{markov_flow_coag} and \eqref{markov_flow_fire}. There is a FFF corresponding to the limit object as well: for the solution of the critical
forest fire equations \eqref{smolu2}+\eqref{nogel} (the uniqueness of the solution is stated in Theorem \ref{theorem:uniqueness}) we define
$\big(\vct{q}(\cdot), \vct{r}(\cdot) \big)$ by
\begin{equation}\label{flow_from_critical_solution}
 \dot{q}_{k,l}(t)=v_k(t)v_l(t), \quad q_{\infty,k}(t)\equiv q_{\infty,\infty}(t) \equiv 0, \quad r_{k}(t)\equiv 0,
 \quad \dot{r}_{\infty}(t)=\varphi(t)
 \end{equation}
with the $\varphi(t)$ of \eqref{v1_evolution}. In Definition \ref{def:flow_topology} we define a topology on the space of FFFs. In later sections we are going to prove that
\[ \big(\vct{q}_n(\cdot), \vct{r}_n(\cdot) \big) \toprob \big(\vct{q}(\cdot), \vct{r}(\cdot) \big)\]
 from which Theorem \ref{theorem:weak_conv_to_unique} will follow.

Summing \eqref{floweqs} for $k \in \N$ we obtain a formula for the evolution of
$\theta(\cdot)$ defined in \eqref{def_V_configurations}: for $s \leq t$
\begin{multline}
\label{ev_of_v_infty_in_flow}
\theta(t)=\theta(s)+\lim_{K \to \infty} \sum_{k=1}^K
\sum_{l=K-k+1}^{\infty} k \cdot \left(
q_{k,l}(t)-q_{k,l}(s)\right) +
\\
\sum_{k=1}^{\infty} k \cdot
\left(q_{k,\infty}(t)-q_{k,\infty}(s)\right) - \left( r_{\infty}(t)-r_{\infty}(s)
\right)
\end{multline}

Later we will see that the term $\lim_{K \to \infty} \sum_{k=1}^K
\sum_{l=K-k+1}^{\infty} k \cdot \left(
q_{k,l}(t)-q_{k,l}(s)\right)$ does not vanish for the FFF defined by \eqref{flow_from_critical_solution} for the unique solution $\vct{v}(t) $ of \eqref{smolu2}+\eqref{nogel} if $T_{\text{gel}}\leq s <t$: this phenomenon is a sign of self-organized criticality.

If $\big(\vct{q}(\cdot), \vct{r}(\cdot) \big)$ is a FFF then the functions $q_{k,l}$, $q_k$, $q$, $r_k$ and $r$ (where
$k,l \in \barn$) are continuous from the left and increasing with initial condition $0$: such functions are the distribution functions of nonnegative measures on $[0,T]$. By $q(T)<+\infty$ and $r(T)<+\infty$ these measures are finite. We denote by "$\weak$" the weak convergence of measures on $[0,T]$, which can alternatively be defined by point-wise convergence
of the distribution functions at the continuity points of the
limiting function.

\begin{definition}
\label{def:flow_topology}
Let $\big(\vct{q}_n(\cdot), \vct{r}_n(\cdot)\big)=
\big( \left(q_{n,k,l}(\cdot)\right)_{k,l \in \barn}, \left(r_{n,k}(\cdot)\right)_{k \in \barn}\big)$, $n=1,2,\dots$ be a sequence of FFFs. Define $q_{n,k}(\cdot)$, $q_n(\cdot)$ and $r_n(\cdot)$ for all $n$ by \eqref{def_of_q_r}.

We say that
$\big(\vct{q}_n(\cdot), \vct{r}_n(\cdot) \big) \to
\big(\vct{q}(\cdot), \vct{r}(\cdot) \big)$ as $n \to \infty$ if
\begin{align*}
\forall \; k,l \in \N \quad q_{n,k,l}(\cdot) &\weak
q_{k,l}(\cdot) \\
\forall\; k \in \N \quad q_{n,k}(\cdot) &\weak
q_k(\cdot) \\  q_n(\cdot) &\weak q(\cdot) \\
\forall\; k \in \N \quad r_{n,k}(\cdot) &\weak r_k(\cdot)\\
 r_n(\cdot) &\weak
r(\cdot)
\end{align*}
\end{definition}
Note that we do not require
 $r_{n,\infty}(\cdot) \weak r_{\infty}(\cdot)$ and $q_{n,k,\infty}(\cdot) \weak q_{k,\infty}$ for $k \in \barn$. Nevertheless these "missing" ingredients of the limit flow $\big(\vct{q}(\cdot), \vct{r}(\cdot) \big)$ of convergent flows are uniquely determined by the convergent ones if we rearrange the relations \eqref{def_of_q_r}:
\begin{align}
 \label{q_k_infty_def_lim} q_{k,\infty}(t) &:= q_k(t)-\sum_{l \in
\N} q_{k,l}(t),
\\[8pt]
 \label{r_infty_def_lim} r_{\infty}(t) &:= r(t)-\sum_{k \in \N}
r_k(t),
\\[8pt]
\label{q_infty_infty_def_lim} q_{\infty,\infty}(t) &:= q(t)-2
\sum_{k \in \N} q_k(t) +\sum_{k,l \in \N} q_{k,l}(t).
\end{align}
In fact, $r_{n,\infty}(\cdot) \not \weak r_{\infty}(\cdot)$ and $q_{n,k,\infty}(\cdot) \not \weak q_{k,\infty}$ have a physical meaning in the forest fire model if $\big(\vct{q}_n(\cdot), \vct{r}_n(\cdot)\big)$ is defined by \eqref{markov_flow_coag} and \eqref{markov_flow_fire}:
\begin{itemize}
\item In the $\lambda(n)=\Ordo(n^{-1})$ regime
   $0 \equiv q_{n,k,\infty}(\cdot) \not \weak  q_{k,\infty}(\cdot) \not \equiv 0$ indicates the presence of a giant component. The precise formulation of this fact for the Erd\H os-R\'enyi model is \eqref{ER_giant_flow_integral}.
\item If
 $\lambda(n) \ll 1$ then  only "large" components burn. Indeed in Proposition \ref{propo:weaklimit}  we are going to prove  that for all $k \in \N$ $r_{n,k}(\cdot)$ converges to $0$ in probability an $n \to \infty$.
 Thus by \eqref{r_infty_def_lim}  we have $r(\cdot)=r_{\infty}(\cdot)$ in the limit. But Theorem \ref{theorem:uniqueness}, Theorem \ref{theorem:weak_conv_to_unique} and \eqref{flow_from_critical_solution} imply that
   $ 0 = r_{n,\infty}(t) \not \weak r_{\infty}(t) =\int_{0}^t \varphi(s)\, ds>0 $ for $t > T_{\text{gel}}$.
\end{itemize}

$\cF_{\vct{v}}[0,T]$ endowed with the topology of Definition \ref{def:flow_topology} is a complete separable metric space:

\begin{lemma}
\label{lemma:polish}
If $\big(\vct{q}_n(\cdot), \vct{r}_n(\cdot) \big) \in \cF_{\vct{v}}[0,T]$
for all $n \in \N$ and
$\big(\vct{q}_n(\cdot), \vct{r}_n(\cdot) \big) \to \big(\vct{q}(\cdot),
\vct{r}(\cdot) \big)$,
then $\big(\vct{q}(\cdot), \vct{r}(\cdot) \big) \in \cF_{\vct{v}}[0,T]$.
\end{lemma}

\begin{proof}
By the definition of weak convergence, $q_{k,l},q_{k},q,r_k,r$ are increasing
left-continuous functions with initial value $0$. We need to check that the functions
$r_{\infty}$, $q_{k,\infty}$, and
$q_{\infty,\infty}$ (defined by \eqref{r_infty_def_lim}, \eqref{q_k_infty_def_lim} and \eqref{q_infty_infty_def_lim}, respectively) are increasing. We may assume that $0 \leq s \leq t \leq T$
 are continuity points of $q_{k,l}$, $q_k$, $q$, $r_k$ and $r$ for all $k,l \in \barn$.

 By Fatou's lemma we get
\begin{align*}
r_{\infty}(t)-r_{\infty}(s)
&=
\lim_{n \to \infty} \left(r_n(t)-r_n(s)\right)-
\sum_{k \in \N} \lim_{n \to \infty}
\left(r_{n,k}(t)-r_{n,k}(s)\right)
\\[8pt]
&\geq
\limsup_{n \to \infty}
\left( r_n(t)-r_n(s)-\sum_{k \in \N}
\left(r_{n,k}(t)-r_{n,k}(s)\right)
\right)
\\[8pt]
&=
\limsup_{n \to \infty}
\left( r_{n,\infty}(t)-r_{n,\infty}(s)\right)
\geq 0.
\end{align*}
One can prove similarly that $q_{k,\infty}$ is increasing for $k \in \N$.
In order to prove that
\[
q_{\infty, \infty}(t)-q_{\infty, \infty}(s)  \geq
\limsup_{n \to \infty}\left( q_{n,\infty,\infty}(t)-
q_{n,\infty,\infty}(s) \right)
\]
let
$\alpha_{n,k,l}:=q_{n,k,l}(t)-q_{n,k,l}(s)$ for $k,l \in \barn$.
By \eqref{q_infty_infty_def_lim}
we only need to check
\begin{equation}\label{fapipa456}
\lim_{n \to \infty} \sum_{k,l \in \barn} \alpha_{n,k,l}-
\limsup_{n \to \infty} \alpha_{n,\infty,\infty}
 \geq 2 \sum_{k \in \N} \lim_{n \to \infty} \sum_{l \in
\barn} \alpha_{n,k,l} - \sum_{k,l \in \N} \lim_{n \to \infty}
\alpha_{n,k,l}.
\end{equation}
Let
\[
K_m:= \{ (k,l): \left( k \geq m \text{ and } l=m \right) \text{ or }
\left( l \geq m \text{ and } k=m \right) \}\cup\{(m,\infty)\}\cup
\{(\infty,m) \}.
\]
The left hand side of \eqref{fapipa456} is $\liminf_{n \to \infty} \sum_{m \in
\N} \beta_{n,m}$, the right hand side is $\sum_{m \in \N} \lim_{n
\to \infty} \beta_{n,m}$, where $\beta_{n,m}:= \sum_{(k,l) \in
K_m} \alpha_{n,k,l}$,
and the inequality follows from Fatou's lemma.

Now that we have proved that the limit of convergent flows is
itself a flow, we only need to check that the limit flow is consistent
with the initial condition $\vct{v}$, but this follows from the facts
that $\cE_{\vct{v}}[0,T]$ is a closed metric space and the mapping from
$\cF_{\vct{v}}[0,T]$  to $\cE_{\vct{v}}[0,T]$ defined by \eqref{floweqs}
is continuous with respect to the corresponding topologies.

\end{proof}

Finally we define the space of all FFF-s as follows:
\[
\cD[0,T]
:=
\big\{
\big(\vct{v}, \vct{q}(\cdot),  \vct{r}(\cdot) \big)
\,:\,
\vct{v}\in\cV, \,\,\,
\big(\vct{q}(\cdot),  \vct{r}(\cdot) \big) \in
\cF_{\vct{v}}[0,T]
\big\}.
\]
This space is again a  complete and separable metric space if we
define $\big(\vct{v}_n, \vct{q}_n(\cdot),  \vct{r}_n(\cdot) \big)
\to  \big(\vct{v}, \vct{q}(\cdot),  \vct{r}(\cdot) \big)$ by
requiring $\vct{v}_n \to \vct{v}$ (coordinate-wise) and $\big(
\vct{q}_n(\cdot),  \vct{r}_n(\cdot) \big) \to \big(
\vct{q}(\cdot),  \vct{r}(\cdot) \big)$.

\begin{lemma}
\label{lemma:compactness}
For any $C<\infty$ the subset
\[
\cK_C[0,T]:=
\big\{
\big(\vct{v}, \vct{q}(\cdot),  \vct{r}(\cdot)
\big)\in\cD[0,T] \,:\,
q(T)\le C
\big\}
\]
is compact in $\cD[0,T]$.
\end{lemma}

\begin{proof}
\[
\lim_{K \to \infty}\left[ \frac{1}{2} \sum_{k=1}^{K}
\sum_{l=1}^{k-1}q_{l,k-l}(T) -\sum_{k=1}^K  q_{k}(T) \right] = -\frac12 q(T)
+\frac{1}{2}q_{\infty,\infty}(T)
\]
by $q(T)\leq C$, dominated
convergence and $q_{k,l}=q_{l,k}$.  Thus summing the equations
\eqref{floweqs} with coefficients $\frac{1}{k}$ we get
\begin{equation*}
\label{flow_equality_1}
\sum_{k=1}^{\infty}\frac{1}{k}v_k(T)-\sum_{k=1}^{\infty} \frac{1}{k} v_k(0) +
\frac12 q(T) = \sum_{k=2}^{\infty} \frac{k-1}{k} r_k(T) +r_{\infty}(T) +
\frac{1}{2}q_{\infty,\infty}(T).
\end{equation*}
The inequalities
\begin{equation}
\label{bounds_on_flow}
r(T) \leq 2 + C, \quad r_{\infty}( T) \leq 1 +
 \frac{1}{2} C, \quad r_k(T) \leq
(1+\frac{C}{2}) \frac{k}{k-1}
\end{equation}
follow from
$\vct{v}(T) \in \cV$ and  $q(T)\leq C$.

By Helly's selection theorem and a diagonal argument we can choose
a convergent subsequence from any sequence of elements of $\cK_C[0,T]$ with
 the limiting FFF itself being an element of  $\cK_C[0,T]$.

\end{proof}

\subsection{The Markov process}
\label{subsection:Markov_process}
It is easy to see that in order to prove Theorem
\ref{theorem:weak_conv_to_unique} we do not need to know anything
about the graph structure of the connected components: by the mean
field property of the dynamics the stochastic process
$\vct{v}_n(t)$ defined by \eqref{def_of_v_k_in_markov} and
\eqref{def_of_vct_v_in_markov} is itself a Markov chain.

The state space of the Markov chain $t\mapsto \vct{V}_n(t)$ is:
\[
\Omega_n
:=
\big\{\vct{V}=\big(V_k)_{k\in\N}\,:\,
V_k\in \{0,k,2k,\dots\},\,\, \sum_{k\ge1}V_k=n\big\}
\]

The allowed jumps of the Markov chain are described by the following
jump transformations for $i \leq j$:
\begin{align*}
&
\sigma_{i,j}:\big\{\vct{V}\in\Omega_n\,:\,
V_i\big(V_j-j\ind{i=j}\big)> 0\big\}\to\Omega_n,
\\[8pt]
&\hskip6cm
\big(\sigma_{i,j}\vct{V}\big)_k
:=
V_k-i\ind{k=i}-j\ind{k=j}+(i+j)\ind{k=i+j},
\\[16pt]
&
\tau_i: \big\{\boV\in\Omega_n:V_i>0\big\}\to\Omega_n,
\hskip9.3mm
\big(\tau_{i}\vct{V}\big)_k
:=
V_k+i\ind{k=1}-i\ind{k=i}
\end{align*}
The corresponding jump rates are $a_{n,i,j},
b_{n,i}:\Omega_n\to\R_+$:
\[
a_{n,i,j}(\vct{V}):= \left((1+ \ind{i=j}) n \right)^{-1} V_i
\big(V_j - j \ind{i=j}\big), \qquad b_{n,i}(\vct{V}):= \lambda(n)
V_i.
\]
The infinitesimal generator of the chain is :
\[
L_nf(\vct{V})= \sum_{i \leq j} a_{n,i,j}(\vct{V})
\big(f(\sigma_{i,j}\vct{V})-f(\vct{V})\big) + \sum_{i}
b_{n,i}(\vct{V}) \big(f(\tau_{i}\vct{V})-f(\vct{V})\big).
\]

We denote by $Q_{n,k,l}(t)$ and by $R_{n,k}(t)$ the number of
$\sigma_{k,l}$-jumps, respectively  $k$-times the number of
$\tau_k$-jumps occurred in the time interval $[0,t]$:
\begin{align}\label{def_of_Qnkl}
Q_{n,k,l}(t) & := \big(1+\ind{k=l}\big)\cdot \abs{\big\{s\in[0,t]\,:\,
\vct{V}_n(s+0)=\big(\sigma_{k,l}\vct{V}_n\big)(s-0)\big\}},
\\[8pt]
\label{def_of_Rnk}  R_{n,k}(t) & := \ind{k\not=1} k \cdot \abs{\big\{s\in[0,t]\,:\,
\vct{V}_n(s+0)=\big(\tau_{k}\vct{V}_n\big)(s-0)\big\}}.
\end{align}
Finally, the scaled objects are
\begin{align}
\label{markov_flow_weights}
v_{n,k}(t)
&:=n^{-1}V_{n,k}(t),
&&
&& \vct{v}_n(t) :=
\big(v_{n,k}(t)\big)_{k\in\N},
\\[8pt]
\label{markov_flow_coag}
q_{n,k,l}(t)
&:= n^{-1}Q_{n,k,l}(t),
&&
q_{n,k,\infty}(t)\equiv 0,
&& \vct{q}_n(t) :=
\big(q_{n,k,l}(t)\big)_{k,l\in\barn},
\\[8pt]
\label{markov_flow_fire}
r_{n,k}(t)
&:= n^{-1}R_{n,k}(t),
&&
r_{n,\infty}(t)\equiv 0,
&& \vct{r}_n(t) :=
\big(r_{n,k}(t)\big)_{k\in\barn}
\end{align}
 Now, given $T\in(0,\infty)$ and
some initial conditions $\vct{v}_n(0)=\vct{v}_n\in\cV_1$, clearly
$t\mapsto\vct{v}_n(t)\in\cV_1$ is a conservative FFE, generated by
the FFF $\big(\vct{v}_n,\vct{q}_n(\cdot),\vct{r}_n(\cdot)\big)
\in\cD[0,T]$ through \eqref{floweqs}. We denote by $\P_{n}$ the
probability distribution of this process on $\cD[0,T]$. We will
always assume that the initial conditions converge, as
$n\to\infty$, to a deterministic element of $\cV_1$:
\begin{equation}
\label{inilim}
\lim_{n\to\infty}v_{n,k}(0)=v_k, \qquad
\vct{v}:=(v_k)_{k\in\N}\in\cV_1.
\end{equation}

\begin{proposition}
\label{propo:weaklimit}
The sequence of probability measures
$\P_{n}$ is tight on $\cD[0,T]$. If $\lambda(n) \ll 1$, then any
weak limit point $\P$  of the sequence  $\P_{n}$ is concentrated
on that subset of $\cD[0,T]$ for which the following hold for $k,l
\in \N$:
\begin{align}
\label{smoluk>1}
&
q_{k,l}(t)
=
\int_0^t v_{k}(s) v_{l}(s) ds,
\quad
q_{k}(t)
=
\int_0^t v_k(s) ds,
\quad
q(t)\leq t,
\quad
r_k(t) \equiv 0
\\[8pt]
\label{inicond}
&
\vct{v}(0)=\vct{v}.
\end{align}
\end{proposition}

\begin{proof}
There is nothing to prove about the initial condition \eqref{inicond}:
it was a priori assumed in \eqref{inilim}.

In order to prove the validity of the integral equations
\eqref{smoluk>1}, note first that it is straightforward that the
processes $\wt q_{n,k,l}(t)$, $\lan\wt q_{n,k,l}\ran(t)$, $\wt
q_{n,k}(t)$, $\lan\wt q_{n,k}\ran(t)$, $\wt r_{n,k}(t)$, $\lan\wt
r_{n,k}\ran(t)$, defined below are martingales:
\begin{align*}
\wt q_{n,k,l}(t)
&:=
q_{n,k,l}(t) - \int_0^t v_{n,k}(s)v_{n,l}(s) ds +
\frac{k\ind{k=l}}{n}\int_0^t v_{n,k}(s) ds,
\\[8pt]
\lan \wt q_{n,k,l}\ran(t)
&:=
\wt q_{n,k,l}(t)^2 -
\frac{\ind{k\not=l}+2\ind{k=l}}{n} \int_0^t v_{n,k}(s)v_{n,l}(s)
ds + \frac{2k\ind{k=l}}{n^2}\int_0^t v_{n,k}(s) ds,
\\[8pt]
\wt q_{n,k}(t)
&:=
q_{n,k}(t) - \int_0^t v_{n,k}(s) ds +
\frac{k}{n}\int_0^t v_{n,k}(s) ds,
\\[8pt]
\lan \wt q_{n,k}\ran(t)
&:=
\wt q_{n,k}(t)^2 - \frac{1}{n} \int_0^t
\big(v_{n,k}(s)^2+ v_{n,k}(s)\big)ds + \frac{2k}{n^2} \int_0^t
v_{n,k}(s) ds,
\\[8pt]
\wt q_{n}(t)
&:=
q_n(t) -t + \frac{1}{n}
\int_0^t m_{n,1}(s)ds,
\\[8pt]
\lan \wt q_{n}\ran(t)
&:=
\wt q_n(t)^2 -\frac{1}{n}
\left(t+ \int_0^t \sum_{k=1}^n v_{n,k}(s)^2
ds \right) +\frac{2}{n^2} \int_0^t m_{n,1}(s)ds,
\\[8pt]
\wt r_{n,k}(t)
&:=
r_{n,k}(t) - \lambda(n) k\int_0^t  v_{n,k}(s) ds,
\\[8pt]
\lan \wt r_{n,k}\ran(t)
&:=
\wt r_{n,k}(t)^2 -
\frac{\lambda(n) k^2}{n} \int_0^t  v_{n,k}(s) ds.
\end{align*}
From Doob's
maximal inequality it readily follows that for any $k,l\in\N$ and
$\vareps>0$
\begin{align*}
& \lim_{n\to\infty} \prob{\sup_{0\le t \le T}
\big|\,q_{n,k,l}(t)-\int_0^t v_{n,k}(s)v_{n,l}(s) ds\,\big|
>\vareps}=0,
\\[8pt]
& \lim_{n\to\infty} \prob{\sup_{0\le t \le T}
\big|\,q_{n,k}(t)-\int_0^t v_{n,k}(s) ds\,\big|
>\vareps}=0,
\\[8pt]
& \lim_{n\to\infty} \prob{\sup_{0\le t \le T} \,q_{n}(t)-t \,
>\vareps}=0,
\\[8pt]
& \lim_{n\to\infty} \prob{\sup_{0\le t \le T}
\big|\,r_{n,k}(t)\,\big|
>\vareps}=0.
\end{align*}
Hence \eqref{smoluk>1}. Tightness follows from
\begin{equation}
\label{q_n_upper_bound}
\expect{q_n(T)}\leq T,
\end{equation}
Markov's
inequality and Lemma \ref{lemma:compactness}.
\end{proof}

If we consider the case $\lambda(n) \equiv 0$ (this is the dynamical Erd\H os-R\'enyi model) then
\eqref{limprob}+\eqref{smolu1} follows from Proposition \ref{propo:weaklimit} since \eqref{floweqs} becomes
\[ v_k(t)=v_k(0)+
\frac{k}{2} \sum_{l=1}^{k-1}q_{l,k-l}(t)
-k q_k(t)= v_k(0) + \int_0^t \frac{k}{2} \sum_{l=1}^{k-1} v_l(s)v_{k-l}(s) - k v_k(s)\, ds\]
which is the integral form of \eqref{smolu1}. Plugging \eqref{smoluk>1} into \eqref{q_k_infty_def_lim} we get  for $t>T_{\text{gel}}$
\begin{equation}\label{ER_giant_flow_integral}
 q_{k,\infty}(t)=\int_0^t v_k(s) \theta(s)\, ds>0.
\end{equation}

\subsection{The integrated Burgers control problem}
\label{subsection:integrated_burgers_control_problem}

If $\vct{v}(\cdot) \in \cE_{\vct{v}_0}[0,T]$  is generated by a
FFF satisfying \eqref{smoluk>1} through \eqref{floweqs}, then
\[ r(\cdot)=\sum_{k \in \barn} r_k(\cdot)= \sum_{k=1}^{\infty} r_k(\cdot)+ r_\infty(\cdot)
=\sum_{k=1}^{\infty} 0 + r_{\infty}(\cdot) =r_{\infty}(\cdot)\]
and $\vct{v}(\cdot)$ is a solution of the \emph{controlled
Smoluchowski integral equations} with control function
$r(\cdot)$:
\begin{align}
\label{smoluk2}
& v_k(t)= v_k(0) + \frac{k}{2} \sum_{l=1}^{k-1}
\int_0^t v_{l}(s) v_{k-l}(s) ds - k \int_0^t v_k(s) ds + \ind{k=1}
r(t), \hskip1cm k \in \N
\\[8pt]
\label{inequality_smol}
&
 v_k(t)\geq 0, \quad \sum_{k=1}^\infty v_k(t)\leq 1
\\[8pt]
\label{inicond2}
&
\vct{v}(0)=\vct{v}_0\in\cV_1.
\end{align}

By $q(T)\leq T$, $r_{\infty}(\cdot)=r(\cdot)$ and \eqref{bounds_on_flow} we get
\begin{equation}
\label{papri}
0=r(0) \leq r(s) \leq r(t)
\quad\text{ for }\quad
0\leq s \leq t \leq T, \quad
 r(T)\le 1+\frac{T}{2}.
\end{equation}
Using induction on $k$ one can see that the initial condition
$\vct{v}_0$ and the control function $r(\cdot)$
determines the solution of \eqref{smoluk2}, \eqref{inicond2}
uniquely.

For $\vct{v}\in\cV$ we introduce the generating function
\begin{equation}
\label{gfdef}
V:[0,\infty)\to[-1,0],
\qquad
V(x):=\sum_{k=1}^\infty v_k e^{-kx}-1.
\end{equation}
$x\mapsto V(x)$ is analytic on $(0,\infty)$ and has the following
straightforward  properties:
\begin{equation}
\label{gfprop1}
\lim_{x\to\infty}V(x)=-1,
\qquad
V^{\prime}(x)\le0,
\qquad
V^{\prime\prime}(x)\ge0.
\end{equation}

It is easy to see that if $t\mapsto\vct{v}(t)$ is a solution of
\eqref{smoluk2}, \eqref{inequality_smol}, \eqref{inicond2} then
the corresponding generating functions $t\mapsto V(t,\cdot)$ will
solve the \emph{integrated Burgers control problem}
\begin{align}
\label{burgers2}
&
V(t,x) -V(0,x) +  \int_{0}^t V(s,x) V^{\prime}(s,x)ds
=
e^{-x} r(t),
\\[8pt]
\label{burgers_boundary_ineq}
&
-1\leq V(t,0)\leq 0
\\[8pt]
\label{burinicond2}
&
V(0,x)=V_0(x).
\end{align}
The control function $r(\cdot)$ was defined to be continuous from the left in \eqref{def_of_q_r}, but it need not
be continuous: when $\lambda(n)=n^{-1} \lambda$ then the FFE obtained as the $n \to \infty$ limit  satisfies \eqref{smoluk2}, \eqref{inequality_smol}, \eqref{inicond2}, but the control function $r(\cdot)$ evolves randomly according to the rules \eqref{giant_lightning}, \eqref{giant_jump}:
\[ \condprob{ r(t+dt)=r(t)+\theta(t)}{\cF(t)}=\lambda \theta(t)dt+ \ordo(dt)\]
Thus $r(\cdot)$ is a random step function in this case.

In order to rewrite \eqref{burgers2} as a differential equation we introduce a new time variable $\tau$:
\begin{equation}\label{tau_max_def}
 t(\tau):=\max \{t: t+r(t) \leq \tau \}
 \end{equation}
It is easily seen that
 $t(\tau)$ is increasing and Lipschitz-continuous:
\be \label{alpha_density_radon}
 t(\tau)= \int_0^{\tau} \alpha(s)\, ds \qquad 0 \leq \alpha(\cdot) \leq 1
\ee
 Given a solution $V(t,x)$ of \eqref{burgers2}, \eqref{burgers_boundary_ineq}, \eqref{burinicond2} define
\begin{equation}\label{V_tau_explicit_def}
\rbV(\tau,x):=V(t(\tau),x)+\left( \tau - t(\tau)-r(t(\tau)) \right)e^{-x}
\end{equation}
Then by \eqref{burgers2} we have
\begin{equation}\label{integrated_burgers_tau}
 \rbV(\tau,x)=V(0,x)- \int_0^{t(\tau)} V(s,x)V'(s,x) \, ds + (\tau-t(\tau))e^{-x}.
 \end{equation}

Now we show that for all $\tau \geq 0$, $x>0$ and $t \geq 0$ we have
\begin{align}
\label{burgers_tau}
&
\partial_{\tau} \rbV(\tau,x)=-\rbV(\tau,x)\rbV'(\tau,x) \alpha(\tau)+ (1-\alpha(\tau))e^{-x}
\\[8pt]
\label{burgers_boundary_ineq_tau}
&
-1\leq \rbV(\tau,0)\leq 0
\\[8pt]
\label{burinicond3}
&
\rbV(0,x)=V_0(x)
\\[8pt]
\label{V_tau_V_t_identity}
&\rbV(t+r(t),x)=V(t,x)
\end{align}
First note that the fact
\begin{equation}\label{V_tau_noteq_alpha_zero}
 \rbV(\tau,x)\neq V(t(\tau),x) \qquad \implies \qquad  \alpha(\tau)=0
\end{equation}
follows directly from \eqref{tau_max_def}, \eqref{alpha_density_radon} and \eqref{V_tau_explicit_def}:  if
$r(t_+)\neq r(t)$, then $\alpha(\tau)=0$ for all $t+r(t)< \tau \leq t+r(t_+)$.
The differential equation \eqref{burgers_tau} follows from \eqref{alpha_density_radon}, \eqref{integrated_burgers_tau}
and \eqref{V_tau_noteq_alpha_zero}.
 The boundary inequality
  \eqref{burgers_boundary_ineq_tau} follows from
 \[ -1 \leq V(t(\tau),x)\leq \rbV(\tau,x) \leq
V(t(\tau)_+,x) \leq 0. \]
The initial conditions \eqref{burinicond2} and \eqref{burinicond3} are equivalent, and \eqref{V_tau_V_t_identity}
follows from \eqref{V_tau_explicit_def} and \eqref{tau_max_def}.

From the definition of Lebesgue-Stieltjes integration it follows that for all $t_1 \leq t_2$ we have
\begin{equation} \label{change_of_variables_tau}
\int_{t_1+r(t_1)}^{t_2+r(t_2)} f(t(\tau))(1-\alpha(\tau))\, d \tau = \int_{t_1}^{t_2} f(t)\, dr(t)
\end{equation}
\section{Boundary behavior}
\label{section:boundary_behaviour}

\subsection{Elementary facts about generating functions}
\label{subsection:elementary}

In this subsection we collect some elementary facts about generating
functions, which will be used along the proof of Theorem \ref{theorem:uniqueness} and Theorem \ref{theorem:weak_conv_to_unique}.
For $\vct{v}\in\cV$ we introduce the generating function
$V(x)$ defined in \eqref{gfdef} which has the straightforward
properties listed in \eqref{gfprop1}. It is also easy
to see that for
any $\vct{v}\in\cV$ and any $x>0$
\begin{equation}
\label{trivibounds}
|V^{\prime}(x)|\le  \frac{1}{e} x^{-1} ,
\qquad
V^{\prime\prime}(x)\le \left(\frac{2}{e}\right)^2 x^{-2} ,
\qquad
|V^{\prime\prime\prime}(x)|\le \left(\frac{3}{e}\right)^3  x^{-3}.
\end{equation}

We define the functions
$E:(0,\infty)\to(0,\infty)$,
$E^*:[0,\infty)\to(0,\infty]$,
$E_*:[0,\infty)\to[0,\infty)$
as follows:
\begin{equation} \label{E_def_no_time}
E(x)
:=
-
\frac{V^{\prime}(x)^3}{V^{\prime\prime}(x)},\quad E^*(x):=\sup_{0 < y\le x} E(y),
\quad E_*(x):=\inf_{0 < y\le x} E(y)
\end{equation}

Note that these functions are continuous on their domain of
definition.

\begin{lemma}
\label{lemma:calculus}

Let $\vct{v}\in\cV_1$.

\begin{enumerate}
\item
For any $x>0$
\begin{eqnarray}
\label{aprioribound1}
0
<
&
V(x)V^{\prime}(x)
&
\le
E^*(x).
\end{eqnarray}
\item
If in addition
\begin{equation}
\label{infinitederivative}
V^{\prime}(0)
:=
\lim_{x\to0}V^{\prime}(x)
=-\infty
\end{equation}
then  the following bounds hold
\begin{eqnarray}
\label{trueboundv}
2^{1/2}E_*(x)^{1/2}x^{1/2}
\le
&
-V(x)
&
\le
2^{1/2}E^*(x)^{1/2}x^{1/2}
\\[8pt]
\label{trueboundvprime}
2^{-1/2}E_*(x)E^*(x)^{-1/2}x^{-1/2}
\le
&
-V^{\prime}(x)
&
\le
2^{-1/2}E^*(x)E_*(x)^{-1/2}x^{-1/2}
\\[8pt]
\notag
\label{trueboundvprimeprime}
2^{-3/2}E_*(x)^3 E^*(x)^{-5/2}x^{-3/2}
\le
&
V^{\prime\prime}(x)
&
\le
2^{-3/2}E^*(x)^3 E_*(x)^{-5/2}x^{-3/2}
\\[8pt]
\label{trueboundvvprime}
E_*(x)
\le
&
V(x)V^{\prime}(x)
&
\le
E^*(x).
\end{eqnarray}
\end{enumerate}
\end{lemma}

\begin{proof}
Since $\vct{v}\in\cV_1$ we have $V(0)=0$.
Denote the inverse function of $-V(x)$ by $X(u)$: $X(-V(x))=x$.
Note that
\begin{equation}
\label{inverse_fun_second_deri}
E(x)=\frac{1}{X^{\prime\prime}(-V(x))},
\end{equation}
and thus
\[
X(0)=0,
\qquad
X^{\prime}(0)=-V^{\prime}(0)^{-1},
\qquad
X^{\prime\prime}(u)=E(X(u))^{-1}.
\]
It follows that for $u\in[0,-V(x)]$:
\begin{eqnarray*}
-V^{\prime}(0)^{-1} +  E^*(x)^{-1}u
\le & X^{\prime}(u) & \le
-V^{\prime}(0)^{-1} + E_*(x)^{-1}u,
\\[8pt]
- V^{\prime}(0)^{-1} u +  E^*(x)^{-1} \frac{u^2}{2}
\le & X(u) & \le
- V^{\prime}(0)^{-1} u +  E_*(x)^{-1} \frac{u^2}{2}.
\end{eqnarray*}
Hence, all the bounds of the Lemma follow directly.
\end{proof}

\subsection{Bounds on $E$}
\label{subsection:bounds_on_E}

We assume given a solution of the \emph{integrated Burgers control
problem}: \eqref{burgers2},
\eqref{burgers_boundary_ineq},  \eqref{burinicond2}  with a control function $r(\cdot)$ satisfying
\eqref{papri}.

We fix
$\overline{t}\in(0,\infty)$, $\overline{x}\in(0,\infty)$. All
estimates will be valid uniformly in the domain
$(t,x)\in[0,\overline{t}]\times  [0,\overline{x}]$. The various
constants appearing in the forthcoming estimates will depend only on
the initial conditions $V(0,x)$ and on the
choice of $(\overline{t},\overline{x})$. The notation
\[
A(t,x)\asymp B(t,x)
\]
means that there exists a constant $1< C<\infty$ which depends
only on the initial conditions \eqref{burinicond2} and the choice
of $(\overline{t},\overline{x})$, such that for any
$(t,x)\in[0,\overline{t}]\times[0,\overline{x}]$
\begin{equation}
\label{alulfelul}
 C^{-1}
B(t,x)\le A(t,x)\le C B(t,x).
\end{equation}
The notation
$A(t,x)=\Ordo(B(t,x))$ means that the upper bound of
\eqref{alulfelul} holds.

In the sequel we denote the derivative of functions $f(t,x)$
 with respect to the time and space variables by $\dot f (t,x)$ and $f'(t,x)$, respectively.

First we define the
\emph{characteristics} given a solution of \eqref{burgers2},
\eqref{burinicond2}, \eqref{burgers_boundary_ineq}: for $t\ge0, x>0$ let $[0,t]\ni s\mapsto
\xi_{t,x}(s)$ be the \emph{unique solution} of the integral equation
\begin{equation}
\label{charaint}
\xi_{t,x}(s)=x-V(t,x)(t-s)+\int_s^t(u-s)e^{-\xi_{t,x}(u)}d r(u).
\end{equation}
Existence and uniqueness of the solution of \eqref{charaint}
follow from a simple fixed point argument. Now we prove that (given
$(t,x)$ fixed) $s\mapsto\xi_{t,x}(s)$ is also solution of the
initial value problem
\begin{equation}
\label{charadiff}
\frac{d}{ds}  \xi_{t,x}(s)=: \dot \xi_{t,x}(s)=V(s,\xi_{t,x}(s)), \qquad
\xi_{t,x}(t)=x.
\end{equation}
In order to prove this we define $\rbV(\tau,x)$ by \eqref{V_tau_explicit_def}. Thus
from \eqref{V_tau_noteq_alpha_zero} it follows that that the solution of \eqref{charadiff} satisfies
\begin{equation}\label{charadiff_tau_1}
\frac{d}{d\tau} \xi_{t,x}(t(\tau))= V( t(\tau), \xi_{t,x}(t(\tau))) \alpha(\tau)=
\rbV(\tau, \xi_{t,x}(t(\tau))) \alpha(\tau)
\end{equation}
From this and \eqref{burgers_tau} we get that
\begin{equation*}
\frac{d}{d\tau} \rbV(\tau,\xi_{t,x}(t(\tau)))= \dot{\rbV}(\tau,\xi_{t,x}(t(\tau)))
+\rbV' (\tau,\xi_{t,x}(t(\tau)))\cdot \frac{d}{d\tau} \xi_{t,x}(t(\tau))=
(1-\alpha(\tau))e^{-\xi_{t,x}(t(\tau))}
\end{equation*}
Integrating this and using $\xi_{t,x}(t)=x$ and \eqref{V_tau_V_t_identity} we get for all $\tau_1 \leq t+r(t)$
\begin{equation*}
\rbV(\tau_1, \xi_{t,x}(t(\tau_1)))=V(t,x)-
\int_{\tau_1}^{t+r(t)} (1-\alpha(\tau))e^{-\xi_{t,x}(t(\tau))}\, d \tau
\end{equation*}
Substituting this into the r.h.s. of \eqref{charadiff_tau_1}, integrating and using \eqref{alpha_density_radon} we get for all $\tau_2 \leq t +r(t)$
\begin{equation*}
\xi_{t,x}(t(\tau_2))=x-V(t,x)(t-t(\tau_2))+
\int_{\tau_2}^{t+r(t)} (t(\tau)-t(\tau_2))e^{-\xi_{t,x}(t(\tau))}(1-\alpha(\tau))\, d \tau
\end{equation*}
Now \eqref{charaint} follows from this by substituting $\tau_2=s+r(s)$ and using \eqref{change_of_variables_tau}.

We  define (similarly to \eqref{E_def_no_time})
\begin{align*}
& E(t,x)
:=
-
\frac{\px V(t,x)^3}{\px^2 V(t,x)},\qquad
& E^*(t,x):=\sup_{0< y \le x}E(t,y),\qquad
& E_*(t,x):=\inf_{0< y \le x}E(t,y),\\
& \rbE(\tau,x)
:=
-
\frac{\px \rbV(\tau,x)^3}{\px^2 \rbV(\tau,x)},\qquad
& \rbE^*(\tau,x):=\sup_{0< y \le x}\rbE(\tau,y),\qquad
& \rbE_*(\tau,x):=\inf_{0< y \le x}\rbE(\tau,y).\\
\end{align*}

Differentiating \eqref{burgers_tau} with respect to $x$ we get
\begin{align}\label{V_tau_first_derivative}
\dot{\rbV}'(\tau,x)&=-\rbV'(\tau,x)^2 \alpha(\tau)-
\rbV(\tau,x)\rbV''(\tau,x) \alpha(\tau)-(1-\alpha(\tau))e^{-x}\\
\dot{\rbV}''(\tau,x)&=-3\rbV'(\tau,x)\rbV''(\tau,x)\alpha(\tau)
-\rbV(\tau,x)\rbV'''(\tau,x)\alpha(\tau) +(1-\alpha(\tau))e^{-x}
\end{align}
Using this and \eqref{charadiff_tau_1} we obtain
\begin{equation}
\label{etx_tau_differential}
\frac{d}{d\tau} \rbE(\tau,\xi_{t,x}(t(\tau)))
=
\left( 3
\frac{\rbV^{\prime}(\tau,\xi_{t,x}(t(\tau)))^2}
{\rbV^{\prime\prime}(\tau,\xi_{t,x}(t(\tau)))}
+
\frac{\rbV^{\prime}(\tau,\xi_{t,x}(t(\tau)))^3}
{\rbV^{\prime\prime}(\tau,\xi_{t,x}(t(\tau)))^2}
\right) e^{-\xi_{t,x}(t(\tau))}\, (1-\alpha(\tau)) d\tau
\end{equation}

\begin{lemma}
\label{lemma:E}
If $m_2(0)=\sum_{k=1}^{\infty} k^2 \cdot v_k(0)  <+\infty$, then for any solution of the
integrated Burgers control problem \eqref{burgers2},
\eqref{burinicond2}, \eqref{burgers_boundary_ineq} with a control
function satisfying \eqref{papri} and for
$(t,x)\in[0,\overline{t}]\times(0,\overline{x}]$ we have
\begin{equation}
\label{Ebound}
E(t,x)
\asymp 1
\end{equation}
\end{lemma}

\begin{proof}
$E(0,x)=\rbE(0,x) \asymp 1$ follows from $m_2(0)<+\infty$. For $t \geq 0$ we
use the formula \eqref{etx_tau_differential} to show that
$0 \leq \frac{d}{d\tau} \rbE(\tau,\xi_{t,x}(t(\tau)))\leq 3$. Since
$0 \leq e^{-\xi_{t,x}(t(\tau))} (1-\alpha(\tau)) \leq 1$ by
 \eqref{alpha_density_radon} we only need to show
\begin{equation}
\label{etxbound}
0
\leq
\frac{V^{\prime}(x)^2}{V^{\prime\prime}(x)^2}
(3V^{\prime\prime}(x) +  V^{\prime}(x))
=
3 \frac{V^{\prime}(x)^2}{V^{\prime\prime}(x)} +
\frac{V^{\prime}(x)^3}{V^{\prime\prime}(x)^2}
\le
3\frac{V^{\prime}(x)^2}{V^{\prime\prime}(x)}
\le
3.
\end{equation}

The lower bound follows from
$
3V^{\prime\prime}(x) + V^{\prime}(x)
=
\sum_{k=1}^\infty (3k^2-k)v_k e^{-kx}>0.
$

The upper bound follows from Schwarz's inequality:
\[
\frac{V^{\prime}(x)^2}{V^{\prime\prime}(x)}=
\frac{\left( \sum_{k=1}^{\infty} k\cdot v_k e^{-kx} \right)^2}
{\sum_{k=1}^{\infty}k^2 \cdot v_k e^{-kx} } \leq \sum_{k=1}^{\infty} v_k e^{-kx} \leq m_0 \leq 1.
\]

Integrating \eqref{etx_tau_differential},
 using $0 \leq \frac{d}{d\tau} \rbE(\tau,\xi_{t,x}(t(\tau)))\leq 3$, \eqref{V_tau_V_t_identity},
 \eqref{change_of_variables_tau} and the last inequality in
\eqref{papri} we obtain
\begin{equation*}
\label{E}
E(0,\xi_{t,x}(0))
\le
E(t,x)
\le
E(0,\xi_{t,x}(0))
+
3(t/2+1).
\end{equation*}
Next we observe that $ x\le\xi_{t,x}(0)\le x+t $ by
\eqref{charadiff} and $-1 < V(t,x) \leq 0$.

The last two bounds yield for
$(t,x)\in[0,\overline{t}]\times(0,\overline{x}]$
\[
0
<
E_*(0,\overline{x}+\overline{t})
\le
E(t,x)
\le
E^*(0,\overline{x}+\overline{t})+3(\overline{t}/2 +1)
<\infty.
\]
\end{proof}

\begin{lemma}\label{giant_growth rate}
If $m_2(0)<+\infty$, then for any solution of the integrated
Burgers control problem \eqref{burgers2},
\eqref{burgers_boundary_ineq}, \eqref{burinicond2} with a control
function satisfying \eqref{papri} there is a constant $C^*$ which depends only on the
initial conditions and $T$  such
that for $T_{\text{gel}} \leq t_1 \leq t_2 \leq T$ we have
\begin{equation}
\label{giant_lipchitz} \theta(t_2)-\theta(t_1) \leq C^*
\cdot(t_2-t_1)
\end{equation}
\end{lemma}

\begin{proof}
$\theta(t)=-V(t,0_{+})$. Since $V(t,x)$ arises from \eqref{gfdef},
we assume $-1< V(t,x)\leq 0$, $V'(t,x)<0$
for all $x>0$.


 Let us pick an arbitrary
$\overline{x}>0$. Let $C$ be a constant such that $E(t,x) \leq C$
for $(t,x)\in[0,T]\times(0,\overline{x}]$.

 First we are going to show that \be
\label{V_times_Vprime_bound_if_defected} \forall\; 0\leq t \leq T
\, , 0<x\leq \overline{x} \qquad V'V(t,x):= V'(t,x)V(t,x) \leq C^*:=\max\{1,2
C\}
 \ee
Note that we cannot use \eqref{aprioribound1} here since that
bound uses $V(t,0)=0$. But $V(0,0)=0$ holds, thus
\eqref{V_times_Vprime_bound_if_defected} holds for $t=0$.
From \eqref{burgers_tau} and \eqref{V_tau_first_derivative} we get
\begin{multline*}
\frac{d}{d\tau} \left(\rbV'\rbV( \tau,x) \right) =\\
\left(-2 \rbV(\tau,x)\rbV'(\tau,x)^2-\rbV(\tau,x)^2\rbV''(\tau,x)\right)\alpha(\tau)+
\left( \rbV'(\tau,x)-\rbV(\tau,x) \right) e^{-x}(1-\alpha(\tau)) \leq \\
-\rbV(\tau,x)\rbV'(\tau,x)^2\left( 2- \frac{1}{C} \rbV'\rbV(\tau,x) \right) \alpha(\tau)+
\left( \rbV'(\tau,x)-\rbV(\tau,x) \right) e^{-x}(1-\alpha(\tau))
\end{multline*}
From \eqref{burgers_boundary_ineq_tau} we get
\[ \rbV'\rbV(\tau,x) \geq 1 \quad \implies \quad
\rbV'(\tau,x) \leq \frac{1}{\rbV(\tau,x)} \leq -1 \leq \rbV(\tau,x)\]
Thus by \eqref{alpha_density_radon} we get
\begin{align*}
\rbV' \rbV (\tau,x) \geq 1 \quad &\implies \quad \left( \rbV'(\tau,x)-\rbV(\tau,x) \right) e^{-x}(1-\alpha(\tau)) \leq 0\\
\rbV' \rbV (\tau,x) \geq 2C \quad &\implies \quad -\rbV(\tau,x)\rbV'(\tau,x)^2\left( 2- \frac{1}{C} \rbV'\rbV(\tau,x) \right) \alpha(\tau) \leq 0 \\
\rbV' \rbV (\tau,x) \geq C^* \quad &\implies \quad \frac{d}{d\tau} \left(\rbV'\rbV( \tau,x) \right) \leq 0
\end{align*}
From $\rbV'\rbV(0,x) \leq C^*$ and the last differential
inequality it easily follows by a ``forbidden region''-argument
that $\rbV'\rbV(\tau,x)\leq C^*$ for all
 $0<x<\overline{x}$ and $0\leq \tau \leq T+r(T)$.
 This and \eqref{V_tau_V_t_identity} implies
 \eqref{V_times_Vprime_bound_if_defected}.

 By \eqref{burgers2} and \eqref{V_times_Vprime_bound_if_defected}
 we have
\[
V(t_1,x)-V(t_2,x) \leq \int_{t_1}^{t_2} V(s,x)V'(s,x)ds \leq C^*
\cdot(t_2-t_1)
\]
for every $0<x<\bar{x}$.
Letting $x \to 0_{+}$ implies the claim of the Lemma.
\end{proof}

\subsection{No giant component in the limit}
\label{subsection:nogiant}

The aim of this subsection is to prove the following proposition:
\begin{proposition}
\label{propo:conservative}
If $n^{-1} \ll \lambda(n) \ll 1$ and
$m_2(0)<+\infty$ holds for $\vct{v}(0)$ on the right-hand side of
\eqref{inilim}
 then any weak limit point $\P$  of the sequence
of probability measures $\P_{n}$ is concentrated on the set of
conservative forest fire evolutions:
\begin{equation}
\label{conserv}
\P
\big(\sum_{k=1}^\infty v_k(t) \equiv 1 \big)=1
\end{equation}
\end{proposition}

We are going to prove Proposition \ref{propo:conservative} by
contradiction: in Lemma \ref{lemma:disjoint_intervals_big_mass} we show that if $\theta(\cdot)\not \equiv 0$ in the limit, then there is a positive time interval  such that $\theta(t)$ has a positive lower bound, and that this implies that even in the convergent sequence of finite-volume models, a lot of mass is contained in arbitrarily big components on this interval. Than in subsequent Lemmas  we prove that these big components indeed burn, which produces such a big increase in the value of the burnt mass $r(\cdot)$ that is in contradiction with
$\expect{r(T)} \leq 2+\expect{q(T)} \leq 2+T$.

$ $

By Proposition \ref{propo:weaklimit} the random FFE obtained as a
weak limit point is almost deterministic: \eqref{smoluk2} holds
with a possibly random control function $r(\cdot)$.
Also, by \eqref{smoluk>1} we $\P$-almost surely have $q(t)\leq t$ from which \eqref{papri} follows.
Thus \eqref{Ebound} and
\eqref{giant_lipchitz} hold $\P$-almost surely for the random
flow obtained as a weak limit point with a deterministic constant $C^*$.

\begin{lemma}
\label{lemma:disjoint_intervals_big_mass}
If $\P_n \weak \P$ where $\P$ does not satisfy \eqref{conserv} on
$\lbrack 0, T \rbrack$, then there exist $\rbeps_1$, $\rbeps_2$,
$\rbeps_3>0$ and a deterministic $\spect\in[\rbeps_1,T]$
such that for every $K<+\infty$, every $m<+\infty$ and every
sequence
\[
\spect -\rbeps_1 <\alpha_1<\beta_1<\alpha_2<\beta_2
<\dots<\alpha_m<\beta_m < \spect
\]
there exists an $n_0<+\infty$ such that
for every $n\geq n_0$ and $1 \leq i \leq m$ we have
\begin{equation}
\label{disjoint_intervals_big_mass} \P_n \left(\max_{\alpha_i \leq
t \leq \beta_i} 1-\sum_{k=1}^{K-1} v_{n,k}(t)
>\rbeps_2 \right)
> \rbeps_3.
\end{equation}
\end{lemma}

\begin{proof}
First we prove that if $\P$ does not satisfy \eqref{conserv} then
there exist $\rbeps_1,\rbeps_2,\rbeps_3>0$ and $\rbeps_1 \leq
\spect \leq T$ such that
\begin{equation}
\label{det_time_ineterval_giant} \P \big(\inf_{\spect-\rbeps_1
\leq t \leq \spect}\theta(t) > \rbeps_2\big) > \rbeps_3.
\end{equation}
Since \eqref{conserv} is violated, we have $ \P\big( \sup_{0\leq t
\leq T} \theta(t) > \rbeps\big)>\rbeps $ for some $\rbeps>0$.

Let $L:=\lfloor \frac{2 C^* T}{\rbeps} \rfloor$ and $t_i:=\frac{\rbeps
i}{2C^*} $ for $1 \leq i \leq L$ where $C^*$ is the constant in
\eqref{giant_lipchitz}.
 Since $\theta(0)=0$ we have
\[
\big\{ \sup_{0\leq t\leq T} \theta(t) > \rbeps \big\}
\subseteq
\bigcup_{i=1}^L
\big\{\theta(t_i)> \frac{\rbeps}{2} \big\}
\]
almost surely with respect to $\P$. Thus $\P\big(\theta(\spect)>
\frac{\rbeps}{2}\big) > \frac{\rbeps}{L}$ for some $\spect \in \{
t_1,\dots t_L \}$. Using \eqref{giant_lipchitz} again
\eqref{det_time_ineterval_giant} follows with
$\rbeps_1:=\frac{\rbeps}{4C^*}$, $\rbeps_2:=\frac{\rbeps}{4}$,
$\rbeps_3=\frac{\rbeps}{L}$.

Now given $K$ and the intervals $\lbrack \alpha_i,\beta_i
\rbrack$, $1 \leq i \leq m$ we define the continuous
functionals $f_i: \cD[0,T] \to \R$ by
\[
f_i\left(\vct{v}(0),\vct{q}(\cdot),\vct{r}(\cdot) \right):=
\frac{1}{\beta_i-\alpha_i} \int_{\alpha_i}^{\beta_i}
\big(1-\sum_{k=1}^K v_k(t)\big)dt
\]
where $v_k(t)$ is defined by
\eqref{floweqs}. Thus for all $i$
\[
H_i:= \{ \left( \vct{v}(0),\vct{q}(\cdot),\vct{r}(\cdot)
\right) \in \cD[0,T] : f_i\left(
\vct{v}(0),\vct{q}(\cdot),\vct{r}(\cdot) \right)>\rbeps_2 \}
\]
is an open subset of $\cD[0,T]$ with respect to the topology of
Definition \ref{def:flow_topology}. Thus by the definition of weak
convergence of probability measures we have
\[
\lim_{n \to \infty}
\P_n(H_i) \geq \P(H_i) \geq \P\left(\inf_{\spect-\rbeps_1 \leq t
\leq \spect}\theta(t) > \rbeps_2 \right) > \rbeps_3
\]
from which the claim of the lemma
easily follows.
\end{proof}

\begin{lemma}
\label{lemma:if_big_mass_big_fire_soon}
If $n^{-1} \ll\lambda(n)$ then for every $\rbeps_2>0$ there is a $\rbeps_4>0$
such that for every $\rbt>0$ there is a $K$ and an $n_1$ such that
for all $n \geq n_1$ $1-\sum_{k=1}^{K-1} v_{n,k}(0) \geq \rbeps_2$
implies
\begin{equation}
\label{if_big_mass_big_fire_soon_eq} \E_n \left(r_{n}(\rbt)\right)
\geq \rbeps_4
\end{equation}
\end{lemma}
The proof of Lemma \ref{lemma:if_big_mass_big_fire_soon} will
follow as a consequence of the Lemmas
\ref{lemma:almost_giant_or_lot_burnt} and
\ref{lemma:if_almost_giant_then_big_fire_soon}.

\begin{proof}[Proof of Proposition \ref{propo:conservative}]
We are going to show that if there is a sequence $\P_n$ such that
the weak limit point $\P$ violates \eqref{conserv} then for some
$n$ we have \be\label{contradiction:too_much_fire} \E_n
\left(r_n(T)\right)
> T +2 \ee which is in contradiction with \eqref{q_n_upper_bound}
and \eqref{bounds_on_flow}. In fact, $T+2$ could be replaced with any finite constant in \eqref{contradiction:too_much_fire}, but  $T+2$ is big enough to have a contradiction.

We define $\rbeps_1$, $\rbeps_2$, $\rbeps_3>0$ and $\spect$
using Lemma \ref{lemma:disjoint_intervals_big_mass}.
Next, we define
$\rbeps_4$ using this $\rbeps_2$ and Lemma
\ref{lemma:if_big_mass_big_fire_soon}.
Given these, we choose $\rbt$ be so small that
\[
\left\lfloor \frac{\rbeps_1}{2\rbt} \right\rfloor \rbeps_3 \rbeps_4 >T+2.
\]
We choose $K$ and $n_1$
big enough so that \eqref{if_big_mass_big_fire_soon_eq} holds.
Further on, we fix the intervals $\lbrack \alpha_i,\beta_i \rbrack$, $1
\leq i \leq m=\lfloor \frac{\rbeps_1}{2\rbt} \rfloor$ so that
$\alpha_{i+1}-\beta_i > \rbt$ holds for all $i$
and also $T -\beta_m>\rbt$ holds.
We choose $n_0$ such that
\eqref{disjoint_intervals_big_mass} holds and let
$n:=\max \{n_0,n_1\}$.

Finally, we define the stopping times $\tau_1, \tau_2,\dots, \tau_m$ by
\[
\tau_i:=\beta_i \wedge \min \{ t: t \geq \alpha_i \text{ and } 1-\sum_{k=1}^{K-1} v_{n,k}(t)
\geq \rbeps_2 \}.
\]
We have $\tau_i+\spect \leq \beta_i +\spect < \alpha_{i+1} \leq
\tau_{i+1}$.

Using the strong Markov property,
\eqref{if_big_mass_big_fire_soon_eq} and
\eqref{disjoint_intervals_big_mass}, the inequality
\eqref{contradiction:too_much_fire} follows:
\[
\expect{r_n(T)}\geq
\sum_{i=1}^m \condexpect{r_n(\tau_i+\rbt)-r_n(\tau_i)}{\tau_i
<\beta_i}\prob{\tau_i<\beta_i}\geq m \rbeps_4 \rbeps_3.
\]
\end{proof}
 Lemma \ref{lemma:if_big_mass_big_fire_soon} stated that if initially a lot of mass is contained in big
 components, then in a short time a lot of mass burns. We prove this statement in two steps: in Lemma
\ref{lemma:almost_giant_or_lot_burnt} we prove that if we start with a lot of mass contained in big components,
 then in a short time either a lot of this mass is burnt or the big components coagulate, so a lot of mass is contained in components of size $n^{1/3}$ (the same proof works if we replace the exponent $\alpha=1/3$ by any $0<\alpha<1/2$). Then in Lemma \ref{lemma:if_almost_giant_then_big_fire_soon} we prove that if we start with a lot of components of size $n^{1/3}$ then in a short time a lot of mass burns.

 $ $

We will make use of the following generating function estimates in
the proof of Lemma \ref{lemma:almost_giant_or_lot_burnt}. If
$V(x)$ is defined as in \eqref{gfdef} and if $\vct{v} \in \cV_1$
then for $\varepsilon \leq \frac12$
\begin{eqnarray}
\label{tail_to_gf}
1-\sum_{k=1}^{K-1} v_k \geq \varepsilon
& \implies &
V(1/K) \leq (e^{-1}-1)\varepsilon
\\
\label{gf_to_tail}
V(1/K) \leq -\varepsilon
& \implies &
1-\sum_{k=1}^{\varepsilon K/2} v_k  \geq \varepsilon/4.
\end{eqnarray}

\begin{lemma}
\label{lemma:almost_giant_or_lot_burnt}
There are constants $C_1 <+\infty$, $C_2>0$, $C_3>0$ such that
if
\begin{equation}
\label{bigmass66}
1-\sum_{k=1}^{K-1} v_{n,k}(0) \geq
\rbeps_2
\end{equation}
for all $n$ then
\begin{equation}
\label{almost_giant_or_lot_burnt_eq}
\lim_{n \to \infty}
\prob{ \sum_{k=C_3 \rbeps_2 n^{1/3}}^{n}
v_{n,k}\left(\rbtfin\right)+ r_n\left(\rbtfin\right) \geq C_2
\rbeps_2} =1
\end{equation}
Where $\rbtfin=\frac{C_1}{K\rbeps_2}$.
\end{lemma}

\begin{proof}[Sketch proof]
If we let $n \to \infty$ immediately, we get that the limiting functions $v_1(t),v_2(t),\dots$    solve \eqref{smoluk2}, \eqref{inequality_smol}, \eqref{inicond2} with a possibly random control function $r(t) \equiv r_{\infty}(t)$.

The $n \to \infty$ limit of \eqref{almost_giant_or_lot_burnt_eq} is
\begin{equation}\label{toy}
 \theta \left( \rbtfin \right)+ r\left(\rbtfin\right) \geq C_2
\rbeps_2 \end{equation}
Now we prove that if $\vct{v}(\cdot)$ is a solution of \eqref{smoluk2}, \eqref{inequality_smol}, \eqref{inicond2} then
 $1-\sum_{k=1}^{K-1} v_{k}(0) \geq \rbeps_2$  implies \eqref{toy} with $C_1=4$ and $C_2=\frac{1}{4}$.
  This proof will also serve as an outline of the proof of Lemma \ref{lemma:almost_giant_or_lot_burnt}.

In order to prove \eqref{toy} define $V(t,x)$ by \eqref{gfdef}. Thus $V(t,x)$ solves the integrated Burgers control problem \eqref{burgers2}, \eqref{burgers_boundary_ineq}, \eqref{burinicond2}.

Define $U(t,x):=V(t,x)-r(t)e^{-x}$. Thus $U'(t,x)=V'(t,x)+r(t)e^{-x}$ and by \eqref{burgers2} we have
$\dot U (t,x)=-V(t,x)V'(t,x)$. Define the characteristic curve $\rbch(\cdot)$ by
\begin{equation}\label{charadiff_sketch}
\dot \rbch(t)=V(t,\rbch(t)) \qquad \rbch(0)=\frac{1}{K}
\end{equation}
Let $u(t):=U(t,\rbch(t))-V(0,\frac1K)$. Thus $u(0)=0$,
 and
\begin{multline}\label{chara_u_ineq_sketch}
 \dot u(t)=\dot U(t,\rbch(t))+U'(t,\rbch(t))\dot \rbch(t)=
 -V(t,\rbch(t))V'(t,\rbch(t)) + \\
 \left( V'(t,\rbch(t))+r(t)e^{-\rbch(t)}\right) V(t,\rbch(t))=r(t)e^{-\rbch(t)}V(t,\rbch(t)) \leq 0.
 \end{multline}
 Thus $u(t) \leq 0$, moreover
 \begin{equation}\label{chara_V_sketch}
 V(t,\rbch(t))=V(0,\frac1K) + r(t)e^{-\rbch(t)}+
 u(t) \leq V(0,\frac1K)+r(t),
 \end{equation}
 \begin{equation}\label{charaint_sketch}
\rbch(t)=\frac{1}{K}+\int_0^t
u(s)\,ds +\int_0^t r(s) e^{-\rbch(s)}ds+t
V(0,\frac{1}{K}) \leq
 \frac1K + t\cdot r(t)+t
V(0,\frac{1}{K}).
\end{equation}
By \eqref{tail_to_gf} we have $V(0,\frac1K) \leq -\frac12 \rbeps_2$. In order to prove that
$\theta \left( \rbtfin \right)+ r\left(\rbtfin\right) \geq \frac14 \rbeps_2$
with $\rbtfin=\frac{4}{K \rbeps_2}$ we consider two cases:

If $r \left( \rbtfin \right) \geq \frac14 \rbeps_2$ then we are done.
If $r \left( \rbtfin \right) < \frac14 \rbeps_2$ define $\tau:=\min \{t: \rbch(t)=0\}$.
By \eqref{charaint_sketch} we have
\[ \rbch(\rbtfin) \leq \frac1K + \rbtfin\cdot r(\rbtfin)+ \rbtfin\cdot \left(-\frac12 \rbeps_2 \right) <
 \frac1K +\frac1K- \frac2K=0 \]
Thus $\tau \leq \rbtfin$.  By \eqref{chara_V_sketch} we get
\[ -\theta(\tau)=V(\tau,0)=
V(\tau,\rbch(\tau)) \leq -\frac12 \rbeps_2 + \frac14 \rbeps_2  = -\frac14 \rbeps_2\]
Thus $\frac14 \rbeps_2 \leq \theta(\tau) \leq \theta(\tau)+r(\tau)\leq
\theta \left( \rbtfin \right)+ r\left(\rbtfin\right)$ because by \eqref{ev_of_v_infty_in_flow} the function
 $\theta(t)+r(t)$ is increasing.
\end{proof}
To make this proof work for Lemma \ref{lemma:almost_giant_or_lot_burnt}  we have to deal with the fluctuations caused by randomness, combinatorial error terms and the fact that $\lambda(n)$ only disappears in the limit.
\begin{proof}[Proof of Lemma \ref{lemma:almost_giant_or_lot_burnt}]
Given a FFF obtained from a forest fire Markov process by
\eqref{markov_flow_weights},\eqref{markov_flow_coag} and
\eqref{markov_flow_fire},
 define
 \[
 U_n(t,x):= \sum_{k=1}^n
 \left[
 v_{n,k}(0) +
\frac{k}{2} \sum_{l=1}^{k-1}q_{n,l,k-l}(t) - k q_{n,k}(t) -
 r_{n,k}(t)
 \right]
e^{-kx}-1-\lambda(n)
\]
By \eqref{floweqs} we have
\[
U_n(t,x)+r_{n}(t)e^{-x}=\sum_{k=1}^n v_{n,k}(t)e^{-kx}-1-\lambda(n)=:
V_n(t,x)-\lambda(n)
 =: W_n(t,x).
\]
\begin{align*}
 W'(t,x)&=-\sum_{k \geq 1} k\cdot v_{n,k}(t)e^{-kx}
 \\ -\frac12 \partial_x \left( W(t,x)+1+\lambda(n)\right)^2 &=
 \sum_{k \geq 1} \frac{k}{2} \sum_{l=1}^{k-1} v_{n,l}(t) v_{n, k-l}(t) e^{-kx} \\
 W''(t,x)&= \sum_{k \geq 1} k^2 \cdot v_{n,k}(t)e^{-kx} \\
W''(t,2x)&= \sum_{k \geq 1} \left( \frac{k}{2}\right)^2 \cdot \one [ 2\, |\, k]\cdot v_{n, \frac{k}{2}}(t) e^{-kx}
\end{align*}

If $X(t)$ is a process adapted to the
filtration $\cF(t)$, let
\[L\,X(t):=\lim_{dt \to 0}\frac{1}{dt} \condexpect{X(t+dt)-X(t)}{\cF_t}\]

Using the martingales of
Proposition \ref{propo:weaklimit} we get
\begin{multline}\label{U_infinitesimal}
L \, U_n(t,x)=
\sum_{k \geq 1}
 \left[
\frac{k}{2} \sum_{l=1}^{k-1} L\, q_{n,l,k-l}(t) - k\cdot L\, q_{n,k}(t) -
L\, r_{n,k}(t)
 \right] e^{-kx}=\\
 \sum_{k \geq 1}
 \left[
\frac{k}{2} \sum_{l=1}^{k-1}
 \left( v_{n,l}(t)v_{n,k-l}(t) -\frac{l\cdot \one [2l=k] }{n} v_{n,l}(t) \right)- \right. \\
   \left. k\cdot \left( v_{n,k}(t)-\frac{k}{n} v_{n,k}(t)  \right) -
\left( \lambda(n)\cdot k \cdot v_{n,k}(t)  \right)
 \right] e^{-kx}=\\
 -\frac12 \partial_x \left( W(t,x)+1+\lambda(n)\right)^2-\frac{1}{n}W''(t,2x) +\\
  W'(t,x)+
 \frac{1}{n} W''(t,x) + \lambda(n) W'(t,x)=\\
 -W_n^{\prime}(t,x)W_n(t,x)+ \frac{1}{n} \left( W_n^{\prime \prime}
(t,x)-W_n^{\prime \prime}(t,2x) \right)
  \end{multline}

 Given the random
function $W_n(t,x)$ we define the random characteristic curve
$\rbch_n(t)$ similarly to \eqref{charadiff_sketch}:
\begin{equation}
\label{rnd_charadiff}
 \dot{\rbch}_n(t)=W_n(t,\rbch_n(t)),\quad \quad \rbch_n(0):=\frac{1}{K}
\end{equation}
This ODE is well-defined although $W_n(t,x)$ is not continuous in
$t$, but almost surely it is a step function with finitely many
steps which is a sufficient condition to have well-posedness for
the solution of \eqref{rnd_charadiff}. Define
$\rbu_n(t):=U_n(t,\rbch_n(t))-W_n(0,\frac1K)$. Thus $\rbu_n(0)=0$ and
\begin{equation}\label{u_n_explicit}
\rbu_n(t)=W_n(t,\rbch_n(t))-W_n(0,\frac1K)-r_n(t)e^{-\rbch_n(t)}=
V_n(t,\rbch_n(t))-V_n(0,\frac1K)-r_n(t)e^{-\rbch_n(t)}
\end{equation}

 The solution of \eqref{rnd_charadiff} is
\begin{equation}
\label{char_rnd_int}
\rbch_n(t)=\frac1K+\int_0^t
\rbu_n(s)\,ds +\int_0^t r_n(s) e^{-\rbch_n(s)}ds+t
W_n(0,\frac1K)
\end{equation}
Putting together \eqref{U_infinitesimal} and
\eqref{rnd_charadiff} similarly to \eqref{chara_u_ineq_sketch} and using \eqref{trivibounds} we get
\begin{equation}
\label{char_infinitesimal}
 L \, \rbu_n(t) \leq
 \frac{1}{n}
\left( W_n^{\prime \prime} (t,\rbch_n(t))-
W_n^{\prime\prime}(t,2\rbch_n(t))\right)
\leq n^{-1} \cdot \rbch_n(t)^{-2}
\end{equation}
Now $\wt\rbu_n(t)=\rbu_n(t)-\int_0^t  L\, \rbu_n(s)  ds$ is a
martingale and
\begin{multline}
\label{char_variance}
L \, \wt \rbu_n(t)^2
=
\lim_{h \to 0_+}\frac{1}{h}
\condexpect{\big(U_n(t+h,\rbch_n(t))-U_n(t,\rbch_n(t))\big)^2
}{\cF_t} \leq
\\
\frac{1}{2} \sum_{k,l=1}^n
\left(\frac{k+l}{n} e^{-(k+l)\rbch_n(t)}
-\frac{k}{n} e^{-k\rbch_n(t)} -\frac{l}{n}e^{-l\rbch_n(t)}\right)^2
v_{n,k}(t)v_{n,l}(t) n
\\
+ \sum_{l=1}^n \left(\frac{l}{n}
e^{-l\rbch_n(t)}\right)^2 \lambda(n) v_{n,l}(t) n
=
\Ordo \left(\frac{1}{n} W_n^{\prime \prime} (t,\rbch_n(t)) \right)
=\Ordo \left( n^{-1} \cdot \rbch_n(t)^{-2} \right)
\end{multline}
Define the stopping time
\[
\tau_n:=\min\{t:\rbch_n(t)=n^{-\alpha}\} \quad \alpha=1/3.
\]
In fact any $0<\alpha<1/2$ would be just as good to make the right-hand side of
\eqref{char_infinitesimal} and \eqref{char_variance} disappear when $t \leq \tau_n$ and $n \to \infty$.

It
follows from
\eqref{char_variance} and
 Doob's maximal inequality that
 \begin{equation*}
\sup_t \abs{\wt \rbu_n(t \wedge \tau_n \wedge T)}
\weak 0 \quad\text{ as }\quad  n \to \infty
\end{equation*}
By \eqref{char_infinitesimal} we have $\wt \rbu_n(t) + \int_0^t
 n^{-1} \cdot \rbch_n(s)^{-2} \, ds \geq \rbu_n(t)$ thus
\begin{equation}
\label{char_weak}
\sup_t  \rbu_n(t \wedge \tau_n \wedge T)
\weak 0 \quad\text{ as }\quad  n \to \infty
\end{equation}

By \eqref{tail_to_gf} and \eqref{bigmass66} we have
\begin{equation}
\label{rbeps5}
V_n(0,\frac1K)
 \leq
 (e^{-1} -1)\rbeps_2=:-\rbeps_5
\end{equation}

Define the events $A_n$, $B_n$ and the time $\rbtfin_n$ by
\begin{align*}
A_n
&:=
\big\{ \sup_{t\leq \tau_n\wedge T}  \int_0^t \rbu_n(s) ds \leq \frac{1}{K} \big\}
\cap
\big\{\rbu_n(\tau_n\wedge T) \leq \rbeps_5/3 \big\},
\\[8pt]
B_n
&:=
\big\{r_n(\tau_n) \leq \rbeps_5/3\big\},
\\[8pt]
\rbtfin_n
&:=
\frac{3}{K \abs{W_n(0, \rbch_n(0))}} \leq \frac{3}{K\rbeps_5},
\end{align*}
 We
are going to show that that there are constants $C_2,C_3<+\infty$
such that
\begin{equation}
\label{we_show_this}
A_n \subseteq
\big\{ \sum_{k=C_3
\rbeps_2 n^{1/3}}^{n} v_{n,k}\left(\rbtfin\right)+
r_n\left(\rbtfin\right) \geq C_2 \rbeps_2 \big\}
\end{equation}
which, since \eqref{char_weak} implies that $\lim_{n \to \infty}
\prob{A_n}=1$, gives \eqref{almost_giant_or_lot_burnt_eq}.

First we show that
\begin{equation}
\label{ABsubseteq}
A_n \cap B_n \subseteq \{\tau_n \leq \rbtfin_n \}.
\end{equation}
If we assume indirectly that $A_n$, $B_n$ and $\tau_n>\rbtfin_n$ hold
then $\int_0^{\rbtfin_n} \rbu_n(s) ds\leq\frac{1}{K}$,
so by \eqref{char_rnd_int} we get
\[
\rbch_n(\rbtfin_n) \leq
\frac{1}{K} +\frac{1}{K}+ \int_0^{\rbtfin_n}
r_n(s) e^{-\rbch_n(s)}ds+\rbtfin_n W_n(0,\rbch_n(0))
\leq
-\frac{1}{K}+
\rbtfin_n \cdot r_n(\tau_n) \leq 0.
\]
But $\rbch_n(\rbtfin_n) \leq 0$ is in
contradiction with $\tau_n>\rbtfin_n$, thus \eqref{ABsubseteq} holds.

Now, by \eqref{u_n_explicit} we have
$V_n(\tau_n,n^{-1/3})=u_n(\tau_n)+V_n(0,\frac1K)+r_n(\tau_n)e^{-n^{-1/3}}$. Thus by
\eqref{rbeps5}, the definition of $A_n$ and $B_n$ and \eqref{gf_to_tail} we get
\begin{multline*}
A_n \cap B_n
\subseteq \big\{ u_n(\tau_n) \leq \frac{\rbeps_5}{3} \big\} \cap
\big\{ V_n(0,\frac1K) \leq -\rbeps_5 \big\} \cap
\big\{ r_n(\tau_n)e^{-n^{-1/3}} \leq \frac{\rbeps_5}{3} \big\} \subseteq \\
\big\{ V_n(\tau_n,n^{-1/3}) \leq \frac{-\rbeps_5}{3} \big\} \subseteq
\big\{\sum_{k=n^{1/3}{\rbeps_5}/{6} }^n v_{n,k}(\tau_n) \geq
{\rbeps_5}/{12}\big\}
\end{multline*}

Thus we have
\begin{multline*}
A_n \subseteq (A_n \cap B_n)\cup B_n^c
\subseteq
\big\{\sum_{k=n^{1/3}{\rbeps_5}/{6} }^n v_{n,k}(\tau_n) \geq
{\rbeps_5}/{12}\big\}
\cup
\big\{r_n(\tau_n) > \rbeps_5/3\big\} \subseteq
\\
\big\{ \sum_{k=C_3 \rbeps_2 n^{1/3}}^{n} v_{n,k}(\tau_n)+
r_n(\tau_n) \geq C_2 \rbeps_2 \big\}
\end{multline*}
with $C_3=(1-e^{-1})/6$ and $C_2=(1-e^{-1})/12$.
But $\sum_{k=C_3\rbeps_2 n^{1/3}}^{n} v_{n,k}(t)+ r_n(t)$
increases with time, from which \eqref{we_show_this} follows.
\end{proof}

\begin{lemma}
\label{lemma:if_almost_giant_then_big_fire_soon}
There are constants $C_4<+\infty$, $C_5>0$ such that if
\[
\sum_{k=C_3 \rbeps_2 n^{1/3}}^{n} v_{n,k}(0)\geq
{C_2\rbeps_2}/{2}
\]
for all $n$ then with
\begin{equation}
\label{t_bla_def}
\rbtfin_n:=C_4\rbeps_2^{-2}\big(n^{-1/3}\log(n)+(n\lambda(n))^{-1}\big)
\end{equation}
we have
\begin{equation}
\label{bla_bla_7}
\lim_{n \to \infty}
\expect{r_n(\rbtfin_n)} \geq C_5 \rbeps_2.
\end{equation}
\end{lemma}
\begin{remark*}
The upper bound \eqref{t_bla_def} is technical: on one hand it is not optimal,
 on the other hand, for the proof of
Lemma \ref{lemma:if_big_mass_big_fire_soon} we only need $\rbtfin_n \ll 1$ as $n \to \infty$.
\end{remark*}
\begin{proof}
If $v$ is a vertex of the graph $G(n,t)$ let $\cC_n(v,t)$ denote
the connected component of $v$ at time $t$. Denote by $\tau_b(v)$
the first burning time of $v$:
\[ \tau_b(v):= \inf \{ t \, : \,
\abs{\cC_n(v,t_+)}<\abs{\cC_n(v,t_-)} \} \] Of course
$\abs{\cC_n(v,\tau_b(v)_+)}=1$. Define $\bar{n}:= C_3 \rbeps_2 n^{1/3}$ and
 \[ \cH_n(t):= \{ v  \, : \,
 \abs{\cC_n(v,0)} \geq \bar{n} \, \text{ and }\, \tau_b(v)>t  \} \]

Fix a vertex $v \in \cH_n(0)$.
\begin{align*}
c_n(t)&:=\frac{1}{n} \abs{\cC_n(v, (t \wedge \tau_b(v))_- )}\\
w_n(t)&:=
\frac{1}{n}\abs{\cH_n(t)} \\
z_n(t)&:=\frac{1}{n} \sum_{w \in \cH_n(0)}
\ind{\tau_b(w) \leq t}=
w_n(0)-w_n(t)
\end{align*}
 Thus $c_n(t)$ is an increasing process (we "freeze" $c_n(t)$ when it burns).
  We consider the right-continuous versions of the processes $c_n(t), w_n(t), z_n(t)$.
\[w_n(0) \geq {C_2\rbeps_2}/{2}
 =:\rbeps_6.\]
  We are going to prove that there are constants $C_4<+\infty$,
$C_5>0$ such that
\begin{equation}
\label{we_show_this_2}
\lim_{n \to \infty} \expect{z_n(\rbtfin_n)}
\geq C_5 \rbeps_2
\end{equation}
which implies \eqref{bla_bla_7}.

Define the stopping times
\begin{align*}
 \tau_w&:= \inf \{t: w_n(t) < {\rbeps_6}/{2}
\} \\
\tau_g&:= \inf \{t: c_n(t) > {\rbeps_6}/{4} \} \\
\tau&:=\tau_b(v) \wedge \tau_w \wedge \tau_g
 \end{align*}
 Since $v \in \cH_n(0)$ we have
 \[c_n(t) \geq c_n(0) =\frac{\abs{\cC_n(v,0)}}{n} \geq \frac{\bar{n}}{n}
 \]
 If $\cC_n(v,t)$ is
 connected to a vertex in $\cH_n(t)$ by a new edge at time $t$
 then
\[
   c_n(t_+)-c_n(t_-) \geq \frac{\bar{n}}{n}, \quad \log(c_n(t_+))-
   \log(c_n(t_-))\geq \log \left( 1+ \frac{\bar{n}}{n c_n(t_-)} \right) \geq
   \frac{\log(2)\bar{n}}{n c_n(t_-)}
   \]
\begin{multline*}
L \, \log(c_n(t)) \geq
\frac{\log(2) \bar{n}}{n c_n(t)} \lim_{dt \to 0}\frac{1}{dt}
\condprob{c_n(t+dt)-c_n(t)\geq \frac{\bar{n}}{n} }{\cF_t}\geq \\
\frac{\log(2)\bar{n}}{n c_n(t)}\cdot
\frac{1}{n}\abs{\cC_n(v,t)}\left(\abs{\cH_n(t)}-
\abs{\cC_n(v,t)}\right)\ind{t\ \leq \tau_b(v)} \geq
\log(2)\bar{n} \cdot \left( w_n(t)-c_n(t)\right) \ind{t\ \leq \tau_b(v)} \geq \\
 \log(2)\bar{n}\frac{\rbeps_6}{4}\ind{t \leq \tau}= n^{1/3} \frac{\log(2)}{8} \cdot C_2 \cdot C_3 \cdot
 (\rbeps_2)^2\cdot \ind{t \leq \tau}=:n^{1/3} \rbeps_7 \ind{t \leq \tau}
\end{multline*}
Thus $\log(c_n(t))- \rbeps_7 \cdot n^{1/3} (t \wedge \tau)$ is a submartingale. Using the
 optional sampling theorem we get
\[- \rbeps_7 \cdot n^{1/3} \expect{\tau}\geq \expect{\log(c_n(\tau))} - \rbeps_7 \cdot n^{1/3} \expect{\tau} \geq \log(c_n(0)) \geq
 -\log(n)\]
By Markov's
 inequality  we obtain that for some constant $C<+\infty$
\[
\prob{\tau \leq C n^{-1/3} \rbeps_2^{-2} \log(n)} \geq \frac12
\]

If $\tau_g \leq \tau_b(v) \wedge \tau_w$, then $\cC_n(v,\tau_g)
>\frac{\rbeps_6}{4} n $, so $\expect{\tau_b(v) -\tau_g} \leq
(n\lambda(n))^{-1} \frac{4}{\rbeps_6}$, which implies
\[
\prob{\tau_w \wedge \tau_b \leq C n^{-1/3} \rbeps_2^{-2} \log(n) +
C^{\prime} (n\lambda(n))^{-1} \rbeps_2^{-1}} \geq \frac{1}{4}.
\]
for some constant $C^{\prime}$. We define  $\rbtfin$ of
\eqref{t_bla_def} with $C_4:=\max\{C,C^{\prime} \}$. Using the
linearity of expectation we get
\[\expect{z_n(\rbtfin)} = \expect{\frac{1}{n} \sum_{w \in \cH_n(0)}
\ind{\tau_b(w) \leq \rbtfin}} \geq \rbeps_6 \prob{\tau_b(v) \leq
\rbtfin}.
\]
The inequality
 $\ind { \tau_w \leq \rbtfin }\frac{\rbeps_6}{2}  \leq
 z_n(\rbtfin)$ follows from the definition of $\tau_w$.
\[
\frac{1}{4} \leq \prob{\tau_w \wedge \tau_b \leq \rbtfin} \leq
\prob{ \tau_w \leq \rbtfin} + \prob{\tau_b \leq \rbtfin} \leq
\expect{z_n(\rbtfin)} \frac{2}{\rbeps_6}+ \expect{z_n(\rbtfin)}
\frac{1}{\rbeps_6}
\]
From this \eqref{we_show_this_2} follows.
\end{proof}

\section{The critical equation}
\label{section:criteq}

\subsection{Elementary properties}
\label{subsection:elementary_proeprties}

Existence to the solutions of \eqref{smoluk2}, \eqref{inicond2}
with initial condition satisfying $m_2(0)<+\infty$ and
 boundary condition
\begin{equation}
\label{conserv2}
\sum_{k=1}^{\infty}v_k(t)\equiv 1
\end{equation}
 follows as corollary to Propositions
\ref{propo:weaklimit} and \ref{propo:conservative}: indeed for
any initial condition $\vct{v}_0\in\cV_1$ we can prepare a
sequence of initial conditions of the random graph problem such
that \eqref{inilim} holds as $n\to\infty$ (we do not need to
assume convergence of $m_{n,2}(0)$ to $m_2(0)$). If $n^{-1}
\ll \lambda(n) \ll 1$ then any weak limit of the probability
measures $\P_n$  is concentrated on a subset of FFFs which
generate a FFE satisfying \eqref{smoluk2}, \eqref{conserv2}.

Moreover it is easily seen that \eqref{conserv2} implies that
$r(\cdot)$ must be continuous, and
for $k\ge 2$, the functions $t\mapsto v_k(t)$ solving
\eqref{smoluk2} are differentiable. Thus $\vct{v}(\cdot)$ solves
\eqref{smolu2}, \eqref{nogel}.

Note that assuming that $\vct{v}(\cdot) \in \cE_{\vct{v}_0}[0,T]$
is a solution of \eqref{smolu2},\eqref{nogel}
 one can deduce only from these equations that
\eqref{smoluk2} holds with a control function $r(\cdot)$
satisfying \eqref{papri}: one has to define a FFF using
\eqref{smoluk>1} and $q_{k,\infty}(\cdot) \equiv 0$: plugging
$\theta(t)\equiv 0$ into \eqref{ev_of_v_infty_in_flow} we can see
that the function $r(\cdot)$ is increasing.

Taking the generating function of a solution of \eqref{smoluk2},
\eqref{inicond2}, \eqref{conserv2} with initial condition
satisfying $m_2(0)<+\infty$ we get a solution of \eqref{burgers2},
\eqref{burinicond2} satisfying the boundary condition $V(t,0)\equiv0$.

In this case
 the increasing function $t\mapsto r(t)$ is absolutely
continuous with respect to Lebesgue measure: its Radon-Nykodim
derivative  $\dot r(t)=\varphi(t)$ is a.e. bounded in
compact domains:

Taking the limit $x\to0$ in \eqref{burgers2} and using
\eqref{Ebound}, \eqref{aprioribound1} (which holds because
  $V(t,0)\equiv0$)    we find
\begin{equation}
\label{pint}
r(t_2)-r(t_1) = \lim_{x\to0}\frac12
\int_{t_1}^{t_2}V(s,x)V^{\prime}(s,x) ds \le C\cdot (t_2-t_1).
\end{equation}
Thus in the sequel we assume given a solution of the \emph{critical Burgers control problem}
\begin{align}
\label{burgers2_diff}
& \dot{V}(t,x)=-V'(t,x)V(t,x)+e^{-x} \varphi(t),
\\[8pt]
\label{burboco2}
&
V(t,0)\equiv0
\\[8pt]
\label{burinicond2_diff}
&
V(0,x)=V_0(x)
\end{align}
where $\varphi(t)$ is nonnegative and bounded on $[0,T]$, and $V(t,x)$ is of the form \eqref{gfdef}.

\begin{lemma}
\label{lemma:infinitederivative}
For any solution of
\eqref{burgers2_diff}, \eqref{burinicond2_diff}, \eqref{burboco2} with
$V^{\prime\prime}(0)<+\infty$ and for any $t\ge T_{\text{gel}}$ (see \eqref{def_gel_time}) we have
$V^{\prime}(t,0):=\lim_{x\to0}V^{\prime}(t,x)=-\infty$.
\end{lemma}

\begin{proof}
We actually prove that for any  $\overline{t}<\infty$,
$\overline{x}<\infty$  there exists a constant
$C=C(\overline{t},\overline{x})>0$ such that for any
$(t,x)\in[T_{\text{gel}},\overline{t}]\times(0,\overline{x}]$,
$-V^{\prime}(t,x)\ge C/\sqrt x$.

One can prove the upper bound of \eqref{trueboundv} for all
$V(x)$ satisfying $V(0)=0$ without the
assumption \eqref{infinitederivative} (the same proof works).

 From \eqref{Ebound}
 and the upper bound of \eqref{trueboundv}
it
follows that there exists a constant $\widetilde C<\infty$ such
that for $(t,x)\in[T_{\text{gel}},\overline{t}]\times
(0,\overline{x}]$
\[
E(t,x)^{-1}\le \widetilde C, \qquad\qquad -V(t,x)\le \widetilde
Cx^{1/2}. \qquad
\]
Differentiating with respect to $x$ in \eqref{burgers2_diff} we get
\begin{multline}
\label{kkk}
\frac{d}{dt} \left(-V'(t,x) \right)=
 V'(t,x)^2+V(t,x)V''(t,x)+e^{-x} \varphi(t)=\\
V'(t,x)^2 \cdot \left(1-\frac{V(t,x)V'(t,x)}{E(t,x)}\right)+
e^{-x} \varphi(t) \geq
 V'(t,x)^2 \left(1-\widetilde C^2 x^{1/2}
\cdot \left(-V'(t,x)\right) \right)
\end{multline}




 There exists a $0<\widehat C$ such that for $x\in(0,\overline{x}]$ we have
\begin{equation}\label{minus_V_prime_bound_at_gel_time}
-V^{\prime}(T_{\text{gel}},x)\ge \widehat C/\sqrt{x}
\end{equation}
 by \eqref{trueboundvprime} and \eqref{Ebound},
 since
$V^{\prime}(T_{\text{gel}},0)=-\infty \; \iff
 \; m_1(T_{\text{gel}})=+\infty$ follows from the fact that
 for $t \leq T_{\text{gel}}$ the solutions of \eqref{smolu1}
and \eqref{smolu2}+\eqref{nogel} coincide,
 and it is well-known from the theory of the Smoluchowski
coagulation equations that we have
 \eqref{erdos_renyi_critical_exponent}
 for the solution of \eqref{smolu1}.

From the differential inequality \eqref{kkk} it follows that
\begin{equation}\label{minus_V_prime_diff_ineq}
-V'(t,x) \leq \frac{1}{\widetilde C}x^{-1/2} \qquad \implies \qquad
\frac{d}{dt} \left(-V'(t,x) \right) \geq 0
\end{equation}
Let $C:=\min\{ \widehat C, \widetilde C^{-1} \}$.
 For  $(t,x)\in[T_{\text{gel}},\overline{t}]\times (0,
\overline{x}]$ the inequality
\[
-V^{\prime}(t,x)\ge C/\sqrt{x}.
\]
follows from \eqref{minus_V_prime_bound_at_gel_time} and
\eqref{minus_V_prime_diff_ineq} by a ``forbidden region''-argument.


\end{proof}

Summarizing: from Lemmas \ref{lemma:calculus}, \ref{lemma:E},
\ref{lemma:infinitederivative} and \eqref{pint} it
follows

\begin{lemma}
\label{lemma:summ}
For
$(t,x)\in[T_{\text{gel}},\overline{t}]\times(0,\overline{x}]$
\begin{align}
\label{vbound}
-V(t,x) & \asymp x^{1/2},
\\[8pt]
\label{vprimebound}
-V^{\prime}(t,x) & \asymp x^{-1/2},
\\[8pt]
\label{vprimeprimebound}
\phantom{-} V^{\prime\prime}(t,x) &
\asymp x^{-3/2},
\\[8pt]
 \label{vvprimebound} \phantom{-} V(t,x) V^{\prime}(t,x) & \asymp
1,
\\[8pt]
\label{pdotbound}
\phantom{-} \varphi(t) & \asymp 1.
\end{align}
\end{lemma}

\subsection{Bounds on $E^{\prime}$}
\label{subsection:bounds_on_E_prime}

In this subsection we assume given a solution of \eqref{burgers2_diff},
\eqref{burboco2}, \eqref{burinicond2_diff} satisfying
$\abs{V^{\prime\prime\prime}(0,0)}<+\infty$.
All of the results
of the previous subsection are valid for $V(t,x)$.

\begin{lemma}
\begin{equation}
\label{iniEprimebound}
E^{\prime}(T_{\text{gel}},x) = \Ordo(x^{-1/2})
\end{equation}
\end{lemma}

\begin{proof}
 We consider the function $X(t,u)$ defined for every $t$ as in the proof
 of Lemma \ref{lemma:calculus}. $X'''(0,u)=\Ordo(1)$ for
$u \in \lbrack 0 ,\bar{u} \rbrack$ by $m_1(0)>0$ and
$m_3(0)<+\infty$. For $t\leq T_{\text{gel}}$ we have $\varphi(t) \equiv 0$ thus
 $V(t,x)$ satisfies
the Burgers equation
\[
\dot{V}(t,x)+V(t,x)V'(t,x)=0
\]
 from which
\[
X(t,u)=X(0,u)-tu
\]
follows. Differentiating \eqref{inverse_fun_second_deri} with respect to $x$ we get
\[
E'(T_{\text{gel}},x)=E(T_{\text{gel}},x)^2 X'''(0,-V(T_{\text{gel}},x))V'(T_{\text{gel}},x).
\]
Now \eqref{iniEprimebound} follows from \eqref{Ebound} and
\eqref{trueboundvprime}.
\end{proof}
From now on, we consider the solution of \eqref{burgers2_diff},
\eqref{burboco2}, \eqref{burinicond2_diff}
 for $t \geq T_{\text{gel}}$,
that is we assume that $T_{\text{gel}}=0$.

Since the function $r(t)$ is continuous we get that $t(\tau)$ defined by
\eqref{tau_max_def}
 is the inverse
function of $t+r(t)$ which by \eqref{V_tau_explicit_def} implies
$\rbV(\tau,x)\equiv V(t(\tau),x)$.
Integrating \eqref{etx_tau_differential}
and using \eqref{V_tau_V_t_identity}, \eqref{change_of_variables_tau}
 we get
 for $0\le t_1\le t_2<\infty$
\begin{align}
\label{etx}
E(t_2,x)
&=
E(t_1,\xi_{t_2,x}(t_1))+ \int_{t_1}^{t_2}
\big\{ 3
\frac{V^{\prime}(s,\xi_{t_2,x}(s))^2}{V^{\prime\prime}(s,\xi_{t_2,x}(s))}
+
\frac{V^{\prime}(s,\xi_{t_2,x}(s))^3}{V^{\prime\prime}(s,\xi_{t_2,x}(s))^2}
\big\} e^{-\xi_{t_2,x}(s)} \varphi(s)\, ds
\\[8pt]
\label{etx2}
&=
E(t_1,\xi_{t_2,x}(t_1))+ \int_{t_1}^{t_2} \big\{
-3 \frac{E(s,\xi_{t_2,x}(s))}{V^{\prime}(s,\xi_{t_2,x}(s))} +
\frac{E(s,\xi_{t_2,x}(s))^2}{V^{\prime}(s,\xi_{t_2,x}(s))^3}
\big\} e^{-\xi_{t_2,x}(s)} \varphi(s)\, ds.
\end{align}

\begin{lemma}
\label{lemma:cont}
The function $(t,x)\mapsto E(t,x)$ is continuous on the domain
$(t,x)\in[0,\overline{t}]\times[0,\overline{x}]$, and
\begin{equation}
\label{phiisvvprime}
\varphi(t)=\lim_{x\to0}V^{\prime}(t,x)V(t,x)=E(t,0).
\end{equation}
\end{lemma}

\begin{proof}
From \eqref{pdotbound} and \eqref{charaint} it follows that the
characteristic curves $\xi_{t,x}(s)$ are jointly continuous in the
variables $\{(t,x,s): 0\le s \le t, \,\,\, 0\le x\}$. And hence, further on,
from \eqref{etx} and \eqref{etxbound}, by dominated convergence it
follows that $(t,x)\mapsto E(t,x)$ is jointly continuous in $\{(t,x):
0\le t, \,\,\, 0\le x\}$.
Further, from \eqref{trueboundvvprime} it follows that
\[
\lim_{x\to0}V(t,x)V^{\prime}(t,x)=
\lim_{x\to0}E(t,x)=:
E(t,0)
\]
Hence, \eqref{phiisvvprime} follows from \eqref{pint} again by
dominated convergence.
\end{proof}

\begin{lemma}
\label{lemma:holder}
$ $

\begin{enumerate}[(i)]
\item
The function $x\mapsto E(t,x)$ is H\"older-$1/2$ at $x\to0$:
\begin{equation}
\label{EHolder}
E(t,x)=\varphi(t)\big(1+\Ordo(x^{1/2})\big).
\end{equation}
\item
The function $t\mapsto\varphi(t)$ is Lipschitz continuous: there
exists a constant $C<\infty$ (which depends only on the initial
conditions  \eqref{burinicond2_diff}
 and the choice of
$\overline{t}$ such that for any
$t_1,t_2\in[0,\overline{t}]$
\begin{equation}
\label{phiLipschitz}
|\varphi(t_1)-\varphi(t_2)|\le C|t_1-t_2|.
\end{equation}
\end{enumerate}
\end{lemma}

\begin{proof}
(i)
We prove $|E^{\prime}(t,x)|=\Ordo(x^{-1/2})$. In this order we shall
use the following a priori estimates
\begin{align}
\label{xibound}
\xi_{t,x}(s)
&
\asymp \big(x^{1/2}+(t-s)\big)^2
\\[8pt]
\label{xiprimebound}
\xi^{\prime}_{t,x}(s)
:=\px\xi_{t,x}(s)
&
=\Ordo\left( \big(x^{1/2}+(t-s)\big)
x^{-1/2}\right).
\end{align}
Indeed: \eqref{xibound} follows from \eqref{charaint},
\eqref{vbound} and \eqref{pdotbound}, and we get
\eqref{xiprimebound} from \eqref{vprimebound} and from the fact that
characteristics do not intersect (thus $ 0 \leq \xi^{\prime}_{t,x}(s)$)
by  differentiating \eqref{charaint} w.r.t. $x$:
\[ 0 \leq \xi^{\prime}_{t,x}(s)  \leq 1-V^{\prime}(t,x)(t-s) \]
The a priori bound
\begin{equation}
\label{Eprimebound1}
|E^{\prime}(t,x)|=\Ordo(x^{-1}).
\end{equation}
follows from
\[
E^{\prime}(t,x)=-3V^{\prime}(t,x)^2 +E(t,x)
\frac{-V^{\prime\prime\prime}(t,x)}{V^{\prime\prime}(t,x)}=
\Ordo((x^{-1/2})^2)+\Ordo(x^{-1})
\]
by \eqref{vprimebound},
\eqref{Ebound} and
\[
-\frac{x}{2} V^{\prime\prime\prime}(t,x) \leq
\int_{\frac{x}{2}}^x V^{\prime\prime\prime}(y)dy \leq
V^{\prime\prime}(\frac{x}{2})=\Ordo(x^{-3/2})
\]
using both
 the upper and lower bounds of \eqref{vprimeprimebound}.

Differentiating with respect to $x$ in \eqref{etx2} yields
\begin{align}
\label{etxprime2}
E^{\prime}(t,x)
&=
E^{\prime}(0,\xi_{t,x}(0))\xi^{\prime}_{t,x}(0)+
\\[8pt]
\notag
&\quad
+
\int_0^t
\big\{
-3
\frac{E^{\prime}(s,\xi_{t,x}(s))}{V^{\prime}(s,\xi_{t,x}(s))}
+
3
\frac{E(s,\xi_{t,x}(s))V^{\prime\prime}(s,\xi_{t,x}(s))}
{V^{\prime}(s,\xi_{t,x}(s))^2}
\\[8pt]
\notag
&\qquad\qquad
+
2
\frac{E(s,\xi_{t,x}(s))E^{\prime}(s,\xi_{t,x}(s))}
{V^{\prime}(s,\xi_{t,x}(s))^3}
-
3
\frac{E(s,\xi_{t,x}(s))^2V^{\prime\prime}(s,\xi_{t,x}(s))}
{V^{\prime}(s,\xi_{t,x}(s))^4}
\\[8pt]
\notag
&\qquad\qquad
+
3
\frac{E(s,\xi_{t,x}(s))}
{V^{\prime}(s,\xi_{t,x}(s))}
-
\frac{E(s,\xi_{t,x}(s))^2}
{V^{\prime}(s,\xi_{t,x}(s))^3}
\big\}
\xi^\prime_{t,x}(s)
e^{-\xi_{t,x}(s)} \varphi(s)ds.
\end{align}

Next using \eqref{Eprimebound1} bound we estimate the expression
of $E^{\prime}(t,x)$ given in \eqref{etxprime2}. Using
\eqref{Ebound}, \eqref{vprimebound}, \eqref{vprimeprimebound},
\eqref{iniEprimebound}, \eqref{xibound}, and \eqref{xiprimebound}
we conclude that if \eqref{Eprimebound1} holds then actually
\begin{equation}
\label{Eprimebound2}
|E^{\prime}(t,x) |=\Ordo(x^{-1/2}).
\end{equation}
The dominating order is given by the first term (outside the
integral) and the first two terms under the integral on the right
hand side of \eqref{etxprime2}.

Finally, \eqref{EHolder} follows from \eqref{phiisvvprime} and
\eqref{Eprimebound2}.

(ii)
In order to prove \eqref{phiLipschitz} we note that from \eqref{etx}
and \eqref{phiisvvprime} it follows that for $0\le t_1\le t_2\le
\overline{t}$
\begin{align*}
\varphi(t_1)-\varphi(t_2)
&
=
E(t_1,0)-E(t_1,\xi_{t_2,0}(t_1))
\\[8pt]
& \qquad\qquad - \int_{t_1}^{t_2} \big\{ 3
\frac{V^{\prime}(s,\xi_{t_2,0}(s))^2}{V^{\prime\prime}(s,\xi_{t_2,0}(s))}
+
\frac{V^{\prime}(s,\xi_{t_2,0}(s))^3}{V^{\prime\prime}(s,\xi_{t_2,0}(s))^2}
\big\} e^{-\xi_{t_2,0}(s)} \varphi(s)ds
\end{align*}
Hence, by \eqref{EHolder}, \eqref{xibound} and \eqref{etxbound} we
obtain directly \eqref{phiLipschitz}.
\end{proof}

Summarizing again,
from Lemmas \ref{lemma:calculus}, \ref{lemma:E},
\ref{lemma:infinitederivative}, \ref{lemma:cont} and
\ref{lemma:holder}
it follows

\begin{proposition}
\label{propo:summary}
For a solution of \eqref{burgers2_diff}, \eqref{burinicond2_diff}, \eqref{burboco2} with initial condition
satisfying $T_{\text{gel}}=0$, \eqref{Ebound} and
\eqref{iniEprimebound} and for
$(t,x)\in[0,\overline{t}]\times(0,\overline{x}]$
\begin{align}
 \label{vbound2} -V(t,x) & =
\sqrt{2\varphi(t)}x^{1/2}\big(1+\Ordo(x^{1/2})\big),
\\[8pt]
 \label{vprimebound2} -V^{\prime}(t,x) & =
\sqrt{\frac{\varphi(t)}{2}}x^{-1/2}\big(1+\Ordo(x^{1/2})\big),
\\[8pt]
 \label{vprimeprimebound2} \phantom{-} V^{\prime\prime}(t,x) & =
\sqrt{\frac{\varphi(t)}{8}}x^{-3/2}\big(1+\Ordo(x^{1/2})\big),
\\[8pt]
 \label{vvprimebound2} \phantom{-} V(t,x) V^{\prime}(t,x) & =
\varphi(t)\big(1+\Ordo(x^{1/2})\big).
\\[8pt]
 \label{vdotbound2} \phantom{-} \dot V(t,x) & = \Ordo(x^{1/2}),
\\[8pt]
 \label{vdotprimebound2} \phantom{-} \dot V^\prime(t,x) & =
\Ordo(x^{-1/2}),
\\[8pt]
 \label{phibound2} \varphi(t) & \asymp 1, \qquad
|\varphi(t_1)-\varphi(t_2)|\le C|t_1-t_2|.
\end{align}
\end{proposition}

In order to prove \eqref{largekasymp} we need Example (c) of
Theorem 4. of chapter XIII.5 of \cite{feller}. With our notations
each of the relations
\[-V(t,x)\sim x^{1-1/2} \sqrt{2\varphi(t)} \quad \text{and} \quad
\sum_{l=k}^{\infty} v_l(t) \sim \frac{1}{\Gamma(\frac{1}{2})}
k^{1/2-1} \sqrt{2\varphi(t)} \] implies the other.

\subsection{Uniqueness}
\label{subsection:uniqueness}

 We are going to prove
Theorem \ref{theorem:uniqueness}. by proving the uniqueness of
\eqref{burgers2_diff}, \eqref{burinicond2_diff}, \eqref{burboco2}.

\begin{proof}[Proof of Theorem \ref{theorem:uniqueness}]
Assume that  $V(t,x)$ and $U(t,x)$ are two solutions of the
critical  Burgers control problem with the same initial
conditions and with the control functions $\varphi(t)$ and
$\psi(t)$, respectively. Denote
\begin{align}
 \label{sums} & S(t,x):=\frac{V(t,x)+U(t,x)}{2}, &&
\sigma(t):=\frac{\varphi(t)+\psi(t)}{2}, &&
\sqrt{\varrho(t)}:=\frac{\sqrt{\varphi(t)}+\sqrt{\psi(t)}}{2}
\\[8pt]
\label{differences}
&
W(t,x):=\frac{V(t,x)-U(t,x)}{2},
&&
\delta(t):=\frac{\varphi(t)-\psi(t)}{2}.
\end{align}
Then, it is easily seen that that (given $S(t,x)$) $W(t,x)$,
$\delta(t)$ will solve the
linear control problem
\begin{align}
\label{diffprob}
&
\dot W(t,x) + \big( S(t,x) W(t,x)\big)^{\prime} = e^{-x}\delta(t),
\\[8pt]
\label{diffbc}
&
W(0,x)\equiv 0,
\\[8pt]
\label{diffic}
&
W(t,0)\equiv 0.
\end{align}
We assume $S(t,x)$ and $\rho(t)$ given, with the regularity properties
inherited from Proposition \ref{propo:summary}:
\begin{align}
\label{Sbound}
-S(t,x)
&
=
\sqrt{2\rho(t)}x^{1/2}\big(1+\Ordo(x^{1/2})\big),
\\[8pt]
\label{Sprimebound}
-S^{\prime}(t,x)
&
=
\sqrt{\frac{\rho(t)}{2}}x^{-1/2}\big(1+\Ordo(x^{1/2})\big),
\\[8pt]
\label{Sprimeprimebound}
\phantom{-}
S^{\prime\prime}(t,x)
&
=
\sqrt{\frac{\rho(t)}{8}}x^{-3/2}\big(1+\Ordo(x^{1/2})\big),
\\[8pt]
 \label{SSprimebound} \phantom{-} S(t,x) S^{\prime}(t,x) & =
\rho(t)\big(1+\Ordo(x^{1/2})\big).
\\[8pt]
 \label{Sdotbound} \phantom{-} \dot S(t,x) & = \Ordo(x^{1/2}),
\\[8pt]
 \label{Sdotprimebound} \phantom{-} \dot S^\prime(t,x) & =
\Ordo(x^{-1/2}),
\\[8pt]
\label{rhobound}
\rho(t) & \asymp 1, \qquad
|\rho(t_1)-\rho(t_2)|\le C|t_1-t_2|.
\end{align}
We will prove that under these conditions,
the unique solution of the problem
\eqref{diffprob}, \eqref{diffbc}, \eqref{diffic}
is $W(t,x)\equiv 0$,  $\delta(t)\equiv 0$.

First we define the characteristics of the equation \eqref{diffprob}:
these are the curves $[0,t]\ni s\mapsto \zeta_t(s)$ defined by the ODE
\begin{equation}
\label{Schara}
\dot\zeta_t(s) = S(s,\zeta_t(s)),
\qquad
\zeta_t(t)=0,
\qquad
\zeta_t(s)>0 \text{ for } s<t.
\end{equation}

Next we define the functions $[0,t]\ni s\mapsto \beta_t(s)$
\begin{equation*}
\label{betadef}
\beta_t(s):=S^{\prime}(s,\zeta_t(s)).
\end{equation*}
The functions $[0,t]\ni s\mapsto \zeta_t(s)$ and  $[0,t]\ni s\mapsto
\beta_t(s)$  are directly determined by $S(t,x)$ and
from \eqref{Sbound}, \eqref{Sprimebound}, \eqref{Sprimeprimebound} and
\eqref{rhobound}
inherit the following regularity properties to be used later:
\begin{align}
\label{zetabound}
\zeta_t(s)
&=
\phantom{-}
\frac{\rho(t)}{2} (t-s)^2 \big(1+\Ordo(t-s)\big),
\\[8pt]
\label{zetadotbound}
\dot\zeta_t(s)
&=
-
\rho(t) (t-s) \big(1+\Ordo(t-s)\big),
\\[8pt]
\label{zetadotdotbound}
\ddot\zeta_t(s)
&=
\phantom{-}
\rho(t) \big(1+\Ordo(t-s)\big),
\\[8pt]
\label{betabound}
\beta_t(s)
&=
-
(t-s)^{-1} \big(1+\Ordo(t-s)\big),
\\[8pt]
\label{betadotbound}
\dot\beta_t(s) &= - (t-s)^{-2}
\big(1+\Ordo(t-s)\big).
\end{align}

We define $[0,t]\ni s\mapsto \eta_t(s)$ as
\begin{equation*}
\label{etadef}
\eta_t(s):=W(s,\zeta_t(s)),
\end{equation*}
with $W(t,x)$ given in \eqref{differences} being  solution of
\eqref{diffprob}, \eqref{diffbc}, \eqref{diffic}.
Then, for any $t\ge0$,   $\delta(s), \eta_t(s)$, $s\in[0,t]$
solves the ODE (boundary value) control  problem
\begin{equation}
\label{odecontrol}
\dot \eta_t(s) + \beta_t(s) \eta_t(s)
=
e^{-\zeta_t(s)}\delta(s),
\qquad
\eta_t(0)=0=\eta_t(t)
\end{equation}
We will prove that this implies $\delta(t)\equiv 0$. Hence it follows
that $W(t,x)\equiv0$.

On the domain $\{(t,s): 0\le s\le t <\infty\}$ we define the integral
kernel
\begin{equation*}
\label{kernel}
\cK(t,s)
:=
\exp\big\{\int_{0}^{s}\beta_t(u) du - \zeta_t(s) \big\}
=
\frac{t-s}{t}\cL(t,s),
\end{equation*}
defined on the same domain $\{(t,s): 0\le s\le t <\infty\}$, where
\begin{equation*}
\label{lernel}
\cL(t,s)
:=
\exp\big\{
\int_{0}^{s}\big(\beta_t(u)+(t-u)^{-1}\big) du - \zeta_t(s)
\big\}.
\end{equation*}
The ODE control problem \eqref{odecontrol} is equivalent to
\begin{equation}
\label{inteq}
\int_0^t \cK(t,s)\delta(s) ds = 0.
\end{equation}
It is handy to introduce the function
\begin{equation*}
\label{gammadef}
\gamma(t):=\int_0^t\delta(s)(t-s) ds.
\end{equation*}
Then, after two integrations by parts the identity \eqref{inteq} is
transformed into  the eigenvalue problem
\begin{equation}
\label{eveq}
\int_0^t \wh{\cK}(t,s)\gamma(s) ds
=
\gamma(t),
\end{equation}
where
\begin{equation*}
\label{kernel2}
\wh{\cK}(t,s)
:=
\big(\ps\cK(t,t)\big)^{-1}\pss\cK(t,s)
=
\frac{2\ps\cL(t,s)-(t-s)\pss\cL(t,s)}{\cL(t,t)}.
\end{equation*}
Using the regularity properties
\eqref{zetabound},
\eqref{zetadotbound},
\eqref{zetadotdotbound},
\eqref{betabound},
\eqref{betadotbound}
it follows that
\begin{equation}
\label{kernelbound}
\sup_{0\le s< t\le\ot} \abs{\wh{\cK}(t,s)}
<
\infty.
\end{equation}
From \eqref{eveq} and \eqref{kernelbound}, by a Gr\"onwall
argument we get $\gamma(t)\equiv0$ and hence
$\delta(t)\equiv0\equiv W(t,x)$, which proves uniqueness of the
solution of \eqref{burgers2_diff}, \eqref{burinicond2_diff}, \eqref{burboco2}.

\end{proof}

\end{document}